\documentclass[a4paper,11pt]{article}
\usepackage[margin=2cm]{geometry}

\usepackage[english]{babel}

\usepackage{tikz,tkz-tab}
\usepackage{array}
\usepackage{stmaryrd}
\usepackage{amsmath}
\usepackage{amssymb}
\usepackage{mathrsfs}
\usepackage{amsthm}
\usepackage{mathtools}
\usepackage[all]{xy} 
\usepackage{enumitem}

\newtheorem{theorem}{Theorem}[section]
\newtheorem{corollary}[theorem]{Corollary}
\newtheorem{proposition}[theorem]{Proposition}
\newtheorem{lemma}[theorem]{Lemma}

\theoremstyle{definition}
\newtheorem{definition}[theorem]{Definition}

\theoremstyle{remark}
\newtheorem{remark}[theorem]{Remark}
\newtheorem{example}[theorem]{Example}

\newcommand*{\R}{\mathbb R}
\newcommand*{\N}{\mathbb N}

\newcommand*{\ad}{\mathrm{ad}}
\newcommand*{\dsr}{\mathrm{ d_{ SR}}}
\newcommand*{\dsrlift}{\widetilde{\mathrm{ d}}_{\mathrm{ SR}}}
\newcommand*{\e}{\mathrm e}
\newcommand*{\chronexp}{\overrightarrow{\exp}}
\newcommand*{\diff}{\mathop{}\!\mathrm{d}}
\newcommand*{\new}[1]{{\color{red}#1}}

\title{
On the Whitney extension property for continuously differentiable horizontal curves in sub-Riemannian manifolds
}
\author{Ludovic Sacchelli\,\footnote{CMAP, \'Ecole Polytechnique, CNRS, Inria, Universit\'e Paris-Saclay, France}\,, Mario Sigalotti\,$^*$\footnote{Inria, Universit\'e Paris-Saclay}
}

\date{\today}

\begin{document}
 \maketitle

 \begin{abstract}
In this article we study the validity of the Whitney $C^1$ extension property for horizontal curves in sub-Riemannian manifolds that satisfy a first-order Taylor expansion compatibility condition. We first consider the equiregular case, where we show that the extension property holds true whenever a suitable non-singularity property holds for the endpoint map on the Carnot groups obtained by nilpotent approximation. We then discuss the case of sub-Riemannian manifolds with singular points and we show that all step-2 manifolds satisfy the $C^1$ extension property. We conclude by showing that the $C^1$ extension property implies a Lusin-like approximation theorem for horizontal curves on sub-Riemannian manifolds.
\end{abstract}
 

\section{Introduction}

The success of sub-Riemannian geometry in geometric measure theory and nonlinear control is due to its simplicity and flexibility as a modeling tool and to the richness of the phenomena that it exhibits and allows to study. 
Other branches of mathematics, such as hypoelliptic operator theory and rough path theory use sub-Riemannian geometry as a natural underlying structure \cite{IHP-Vol}.

A sub-Riemannian structure on a manifold $M$ is characterized by a distribution $\Delta\subset TM$ endowed with a point-dependent norm which can be used for measuring the length of \emph{horizontal curves}, i.e., absolutely continuous curves which are tangent to $\Delta$. Horizontal curves play a fundamental role in 
sub-Riemannian geometry, since the \emph{sub-Riemannian distance} 
is defined as the minimal length of an horizontal curve connecting two points.

A natural metric property that it makes sense to test on a sub-Riemannian structure is the extendability of regular horizontal curves. 
The Euclidean counterpart of this property is the well-known Whitney extension theorem for a map $\gamma$ from a closed subset $K$ of $\R$ into $\R^d$, $d\in \N$. In this case, the extendability holds under the sole assumption that  the variation of the jets of $\gamma$ on $K$ is compatible with Taylor's expansions.

We study in this paper the counterpart of the Whitney extension theorem for $C^1$ horizontal ($C^1_H$ in the following) curves in sub-Riemannian manifolds. 
A useful intermediate ground where this problem can be set is provided by Carnot groups. Whitney extension theorems for maps between Carnot groups have been the object of research in the past: in particular the case of real-valued $C^1$ functions on the Heisenberg group 
has been considered 
in \cite{franchi_2001_rectifiability}
and extended to the general case of
real-valued  $C^m$ functions on Carnot groups  
in \cite{Vodopyanov_2006_Whitney_Carnot_groups}.
The problem which we consider here is different, since the domain of the map which we seek to extend is contained in $\R$ and the sub-Riemannian  
structure 
is taken on the codomain. 
The latter formulation of the problem has been proposed for Carnot groups by F.~Serra Cassano in \cite{serracassano:IHPVol1}.

The question of extendability of $C^1_H$ curves 
in Carnot groups has been answered positively in the case of the Heisenberg group by S.~Zimmerman in \cite{ZimmermanWhitneyHeisenberg}. For general Carnot groups, it has been proved in \cite{juillet_sigalotti} that the $C^1_H$ extension property holds if and only if the group is \emph{pliable}, that is, for every horizontal vector $v$ the endpoint map $\mathcal{C}_0\ni u\mapsto E(v+u)$
is locally open at $0$, where $\mathcal{C}_0$ is the space of continuous functions $u$ from $[0,1]$ to the horizontal distribution of the Carnot group such that $u(0)=0$ and $E(w)$ is the endpoint 
at time $1$ of the trajectory starting at the identity and tangent to $w$.
This characterization in terms of pliability, together with some tools from geometric control theory, are used in \cite{juillet_sigalotti} to prove that step-2 Carnot groups satisfy the $C^1_H$ extension property. Several examples of non-pliable Carnot groups are also presented.

A closely related subject is the one of Lusin-like approximations 
in Carnot groups, namely the property that an absolutely continuous horizontal curve coincides, out of a set of arbitrarily small measure, 
 with a $C^1_H$ curve. 
 The validity of such a Lusin approximation theorem in the Heisenberg group has been proved in \cite{speight2016LusinCarnot,ZimmermanWhitneyHeisenberg} and extended to the case of step-2 Carnot groups in \cite{le_donne_speight2016LusinStep2}. In \cite{juillet_sigalotti} it is shown that if a Carnot group is pliable then it satisfies the Lusin approximation property.

In this paper we enlarge the analysis from Carnot groups to general sub-Riemannian manifolds (not necessarily equiregular).  
The first step of this program is to provide a suitable definition of the $C^1_H$ extension property in sub-Riemannian manifolds. This is done by showing that $C^1_H$ curves admit an intrinsic first-order Taylor expansion with uniform reminder, evaluated with respect to the sub-Riemannian distance. This can be seen as a form of uniform Pansu-differentiability for $C^1_H$ curves (see also  the results in \cite{Vodopyanov_2006_Whitney_Carnot_groups}, partially recalled in  Section~\ref{S:Lusin}).
The key for investigating this property is the use of 
nilpotent approximations, which characterize the infinitesimal metric structure at a point of the sub-Riemannian manifold. In the equiregular case, nilpotent approximations have a Carnot group structure and one can take advantage of the metric estimates 
given by the celebrated Ball-Box theorem \cite{bellaiche1996TagentSpace}. 
The non-equiregular case can be tackled by desingularization.

The second step of our analysis consists in 
providing sufficient conditions for the $C^1_H$ extension property in sub-Riemannian manifolds to hold true. In the equiregular case, the conditions are expressed in terms of the pliability properties of the nilpotent approximations. More precisely, we prove that
the $C^1_H$ extension property holds if
every Carnot group corresponding to a nilpotent approximation of the sub-Riemannian manifold is \emph{strongly pliable}, that is, for every horizontal vector $v$, 
not only the endpoint map $\mathcal{C}_0\ni u\mapsto E(v+u)$
is locally open at $0$, but also there exists a sequence of points $u_n$ converging to $0$ in $\mathcal{C}_0$ such that $E(v+u_n)=E(v)$ and  $u\mapsto E(v+u)$ is a submersion at $u_n$ for every $n\in\N$. Strong pliability introduces a guarantee of structural stability in the inversibility of the endpoint map, allowing to deduce extendability properties on the sub-Riemannian manifold from those of the nilpotent approximations. 
As a consequence of this condition, we deduce that step-2 sub-Riemannian manifolds satisfy the $C^1_H$ extension property. More generally, second order conditions for the local openness of the endpoint map can be used to express sufficient conditions for strong pliability in terms of the Goh and the generalized Legendre conditions.

Beyond its own metric interest, the $C^1_H$ extension property can be used to characterize rectifiability in sub-Riemannian manifolds: we show that, if the $C^1_H$ extension property  holds true, then rectifiability by Lipschitz curves is equivalent to rectifiability by $C^1_H$ curves.   This equivalence is based on a generalization of Lusin approximation theorem for sub-Riemannian manifolds that satisfy the $C^1_H$ extension property.

\subsection{Article walkthrough}

In Section~\ref{S:SRG} we recall some basic definitions and properties in sub-Riemannian geometry. In particular, 
we discuss two important tools used in the paper, namely the nilpotent approximation of a sub-Riemannian structure and the distance estimates
 given by the Ball-Box theorem. 
We also recollect some basic facts about the chronological exponential notation 
and we use the variation of constant formula to prove a useful 
distance estimate (Lemma~\ref{L:our_distance_estimate}).

Section~\ref{S:Whitney_cond} is dedicated to the study of the $C^1_H$-Whitney condition. Taking inspiration from the uniform Pansu-differentiability of $C ^1_H$ curves (Proposition~\ref{P:Taylor_exp}),  we propose 
a definition of the $C^1_H$-Whitney condition (Definition~\ref{D:C1H_Whitney_cond}) which is proved to be independent of the choice of frame  on   equiregular manifolds (Proposition~\ref{P:real_Px<->Py}).
In order to extend this result to general sub-Riemannian manifolds, we introduce in Section~\ref{SS:Forward_Backward_WC} forward and backward $C^1_H$-Whitney conditions. These asymmetric conditions turn out to be equivalent to the original $C^1_H$-Whitney condition on equiregular manifolds (Proposition~\ref{P:Wg<=>Wd}). 
Forward and backward $C^1_H$-Whitney conditions can be recast in terms of the nilpotent approximations of the sub-Riemannian structure (Propositions~\ref{P:lim_dilation_forward} and \ref{P:lim_dilation_backward}).
 This proves useful in Section~\ref{SS:singular_WC}, where 
it is proved
that the $C^1_H$ extension property is inherited by the projection of an equiregular lift (Corollary~\ref{C:projection_W_T}).

In Section~\ref{S:Strong_pliability}, we propose a sufficient condition for the $C^1_H$ extension property to hold in terms of strong pliability
(Definition~\ref{D:strong_pliability}). 
Strong pliability 
is always satisfied at regular values of the endpoint map 
and can be investigated through second order conditions at critical points (Section~\ref{SSS:second_order}).
In Section~\ref{SS:Theorem}, we use uniform estimates for the nilpotent approximations 
  to prove that strong pliability 
  implies the $C^1_H$ extension property for equiregular sub-Riemannian manifolds (Theorem~\ref{P:Strongly_pliable=>WT}). As a consequence of this result, together with the previous desingularization analysis, we are able to prove that all step-2 sub-Riemannian manifolds  have the $C^1_H$ extension property (Corollary~\ref{C:step_2}).

As a conclusion, we give in Section~\ref{S:Lusin} an application of the $C^1_H$ extension property, proving that it implies the Lusin approximation of horizontal curves (Proposition~\ref{P:Lusin}), which in turns can be used to characterize $1$-rectifiability (Corollary~\ref{C:rectifiability}).

\subsubsection*{Acknowledgments} 
The authors would like to thank Francesco Boarotto and Fr\'ed\'eric Jean for several fruitful discussions. 

This research  has been supported by the ANR SRGI (reference ANR-15-CE40-0018) 
and 
by a public grant as part of the {\it Investissement d'avenir} project, reference ANR-11-LABX-0056-LMH, LabEx LMH, in a joint call with {\it Programme Gaspard Monge en Optimisation et Recherche Op\'erationnelle}.

\section{Sub-Riemannian Geometry}\label{S:SRG}
In this section we   introduce some classical notions   from the fields of sub-Riemannian geometry and control theory that we use in the following sections. For  more details, we refer to the publications cited below, in particular the books \cite{Sachkov2004Control_theory_geometric_viewpoint}, \cite{jean2014control} and \cite{ABB_nov_2016}.
\subsection{Sub-Riemannian manifolds}

We give the definition of a sub-Riemannian manifold as it can be found in \cite{ABB_nov_2016,MR3010287}.
\begin{definition}\label{D:SRmanifold}
Let $M$ be a smooth connected manifold. A \emph{sub-Riemannian structure} on $M$ is a pair $(U,f)$ where 
\begin{enumerate}[label=\roman*.]
\item $U$ is a Euclidean bundle with base $M$ and Euclidean fiber $ U_q$, {\it i.e.}, for every $q\in M$, $U_q$ is a vector space endowed with a scalar product $(\cdot \mid \cdot)_q$ smooth with respect to $q$. In particular, the dimension of $U_q$ is constant with respect to $q\in M$.

\item $f:U\rightarrow TM$ is a smooth map that is a morphism of vector bundles, {\it i.e.}, $f$ is linear on fibers and the  diagram 
$$
\xymatrix{
U \ar[r]^f  \ar[rd]_{\pi_U}  &  TM   \ar[d]^\pi  \\
 & M
}
$$
is commutative (with $\pi_U:U\rightarrow M$ and $\pi:TM\rightarrow  M$ the canonical projections).
\item The set of \emph{horizontal vector fields} $\Delta=\{ f(\sigma)\mid \sigma :M\rightarrow U \text{ smooth section}\}$ is a Lie bracket-generating family of vector fields. 
\end{enumerate}

A \emph{sub-Riemannian manifold} is then a triple $(M,U,f)$ where $M$ is a smooth manifold endowed with a sub-Riemannian structure $(U,f)$. The \emph{distribution} of this manifold is the family of subspaces 
$$
(\Delta_q)_{q\in M} \text{ where } \Delta_q=f(U_q)\subset T_q M,
$$
and $ \dim \Delta_q$ is called the \emph{rank of the sub-Riemannian structure at $q$}.

With an abuse of notation we will sometimes denote the sub-Riemannian manifold $(M,U,f)$ by $(M,\Delta,g)$, with $g$ a quadratic form on $\Delta$ obtained by projection of the Euclidean structure, as explained in the next definition.
\end{definition}

\begin{example}\label{ex:grushin}
Recall that the Grushin plane \cite{Grushin} is a rank-varying sub-Riemannian structure on $\R^2$ having as moving orthonormal frame
$$
X_1=\partial_x,  \qquad X_2=x \partial_y .
$$ 

In terms of Definition~\ref{D:SRmanifold}, such a sub-Riemannian structure is identified with the triple $(\R^2,U,f)$ where
 $ U \simeq \R^2\times \R^2$ is the standard $2$-dimensional Euclidean fiber over $\R^2$, and
$$
\begin{array}{rccc}
f:& U &\longrightarrow& T\R^2\simeq \R^2\times\R^2
\\
& ((x,y),(u,v))&\longmapsto &((x,y),(u, v x)).
\end{array}
$$
\end{example}

\begin{definition}
An absolutely continuous curve $\gamma:I\rightarrow M$  is said to be \emph{horizontal} if 
there exists $u:I\rightarrow U$ measurable and essentially bounded such that $\gamma=\pi_U(u)$ and 
$\dot{\gamma}(t)=f(\gamma(t),u(t))$ for almost every $t\in I$. 
If, moreover, there exists such a function $u$ which is continuous, then $\gamma$ is said to be a \emph{$C^1_H$ curve}.
For $v\in \Delta_q$, set
$$
g( v,v)=\inf\{  (u\mid u)_q\mid f(q,u)=(q,v), (q,u)\in U\}.
$$
We then define the \emph{length of the horizontal curve $\gamma$} as
$$
l(\gamma)=\int_I g(\dot{\gamma}(t),\dot{\gamma}(t))^{1/2} \diff t.
$$
With this length we are able to define the \emph{Carnot-Caratheodory distance between two points $p,q\in M$} as
$$
\dsr (p,q)
=
\inf
\left\{
l(\gamma)
\mid
\gamma:(a,b)\rightarrow M \text{ horizontal},\gamma(a)=p,\gamma(b)=q
\right\}.
$$
\end{definition}

\begin{definition}\label{D:frame}
Let $(M,U,f)$ be a sub-Riemannian manifold and $\Omega$ be an open subset of $M$. We call \emph{frame of the distribution on $\Omega$} a family of horizontal vector fields $(X_1,\dots, X_m)$ such that there exists  a smooth Euclidean frame $(e_1,\dots ,e_m)$ of the Euclidean bundle  on $\pi_U^{-1}(\Omega)$  that satisfies
$$
X_j=f_* e_j, \quad 1\leq j\leq m.
$$
\end{definition}
As a direct consequence of this definition, we have that for any two frames $(X_1,\dots ,X_m)$ and $(Y_1,\dots ,Y_m)$ on $\Omega$, there exists a smooth map $c$ from $\Omega$  to the orthogonal group $\mathrm{O}(m)$
such that 
$$
X_i(q)=\sum_{j=1}^{m} c_{ij}(q) Y_j(q), \quad 1\leq i\leq m, \ q\in \Omega.
$$

\begin{definition}\label{D:lift}
Let $(\widetilde{M},\widetilde{U},\widetilde{f})$ and $(M,U,f)$ be two sub-Riemannian manifolds. We say that $(\widetilde{M},\widetilde{U},\widetilde{f})$ is a \emph{lift} of $(M,U,f)$ if there exists 
$
\phi:\widetilde{U}\rightarrow U
$
a fiberwise isometry
and $\psi:\widetilde{M}\rightarrow M$ a submersion such that the  diagram
$$
\xymatrix{
\widetilde{U} \ar[r]^{\widetilde{f}}  \ar[d]_{\phi}  &  T\widetilde{M}   \ar[d]^{\psi_*}  \\
U\ar[r]^f & TM
}
$$
 is commutative.
\end{definition}

A consequence of the isometry condition   is that for $\dsrlift$ and $\dsr$ the respective sub-Riemannian distances of $(\widetilde{M},\widetilde{U},\widetilde{f})$ and $(M,U,f)$, we have 
\begin{equation}\label{E:proj_lift_dist}
\dsr(\psi(p),\psi(q))\leq \dsrlift(p,q), \qquad \forall p,q\in \widetilde{M}.
\end{equation}
Moreover, any frame $(X_1,\dots, X_m)$ of $(M,U,f)$ admits a lift $(\widetilde{X}_1,\dots ,\widetilde{X}_m)$ ({\it i.e.},  $\psi_* \widetilde{X}_i=X_i$ for all $1\leq i\leq m$) that is a frame of  $(\widetilde{M},\widetilde{U},\widetilde{f})$.

\begin{example}
Consider the standard Heisenberg group structure on $\R^3$, endowed with canonical coordinates $(x,y,z)$ and frame $(X_1,X_2)$ such that
$$
X_1=\partial_x - \frac{y}{2} \partial_z\qquad X_2=\partial_y + \frac{x}{2} \partial_z.
$$ 
Its expression in terms of sub-Riemannian manifold is the following. We set $\widetilde{U}\simeq \R^3\times \R^2$ to be the standard $2$-dimensional Euclidean fiber over $\R^3$, and $T \R^3 \simeq \R^3 \times \R^3$.
By setting
$$
\begin{array}{rccc}
\widetilde{f}:&\widetilde{U}&\longrightarrow& T\R^3
\\
& ((x,y,z),(u,v))&\longmapsto &((x,y,z),(u,v,(v x- uy)/2))
\end{array}
$$
we have that the diagram
$$
\xymatrix{
 \R^3\times \R^2 \ar[r]^{\widetilde{f}}  \ar[rd]_{\widetilde{\pi}_U}  &  T\R^3  \ar[d]^{\widetilde{\pi} } \\
 & \R^3
}
$$
is commutative.
The submersion
$$
\begin{array}{rccc}
\psi:&\R^3&\longrightarrow& \R^2
\\
& (x,y,z)&\longmapsto &(x,z+xy/2),
\end{array}
$$
induces the differential map
$$
\begin{array}{rccc}
\psi_*:&T\R^3&\longrightarrow& T\R^2
\\
& ((x,y,z),(u,v,w))&\longmapsto &((x,z+xy/2),(u,uy/2+vx/2+w)).
\end{array}
$$
As a consequence,
$$
\psi_*\circ \widetilde{f}((x,y,z),(u,v))= ((x,z+xy/2),(u,vx)).
$$
Denoting by $\phi$ the fiberwise isometry
$$
\begin{array}{rccc}
\phi:&\widetilde{U}&\longrightarrow& U\simeq \R^2\times \R^2
\\
& ((x,y,z),(u,v))&\longmapsto &((x,z+xy/2),(u,v))
\end{array}
$$
and by $f$ the smooth bundle morphism of the Grushin plane
$$
\begin{array}{rccc}
f:& U &\longrightarrow& T\R^2
\\
& ((x,y),(u,v))&\longmapsto &((x,y),(u, v x))
\end{array}
$$
(see Example~\ref{ex:grushin}), one easily checks that $\psi_*\circ \widetilde{f}=f\circ \phi$. Hence the Heisenberg group $(\R^3,\widetilde{U},\widetilde{f})$ is a lift of the Grushin plane $(\R^2,U,f)$.
\end{example}

\subsection{Nilpotent approximation}\label{SS:Nilp_approx}

Set $\Delta^1=\Delta$ and $\Delta^{k+1}=\Delta^k+[\Delta^k,\Delta]$ for every integer $k\geq 1$.
At any point $p\in M$, 
the Lie bracket generating condition ensures that there exists an integer $r\geq 1$, that we call \emph{step of the sub-Riemannian structure at $p$}, such that 
\begin{equation}\label{E:flag}
\Delta^1_p\subseteq\Delta^2_p\subseteq\dots \Delta^{r-1}_p\subsetneq\Delta^r_p=T_pM.
\end{equation}
The finite sequence of integers  $(\dim \Delta^1_p, \dots, \dim \Delta^r_p)=(n_1,\dots,n_r)$ is called \emph{growth vector at $p$}. If  the growth vector is constant on a neighborhood of $p$, $p$ is said to be \emph{regular}, and \emph{singular} otherwise.
The manifold itself  is said to be \emph{equiregular} if each of its points is regular and we say that it is singular if it contains singular points.

We call \emph{desingularization of $(M,\Delta,g)$} a lift 
 $(\widetilde{M},\widetilde{\Delta},\widetilde{g})$ of $(M,\Delta,g)$  that is equiregular.
  As shown in \cite[Lemma 2.5]{jean2014control}, at any given point $p\in M$,  there exists a desingularization on some open neighborhood of $p$ that has the same step as $\Delta$ at $p$.

The relation between the flag \eqref{E:flag} and the distance $\dsr$ is characterized by the so-called \emph{weights at $p$}, that is the sequence of integers $w=(w_1,\dots ,w_d)$ such that $w_j=s$ if $n_s<j \leq n_{s+1}$ (with $n_0=0$). Written in full, that is
$$
w=(\underbrace{1,\dots ,1}_{n_1\text{ times}},\underbrace{2,\dots ,2}_{(n_2-n_1)\text{ times}},\dots ,\underbrace{r,\dots ,r}_{(n_{r}-n_{r-1})\text{ times}}).
$$

 A (smooth) system of coordinates $(x_1,\dots, x_d):\Omega\rightarrow \R^d$ is said to be a \emph{system of privileged coordinates at $p$}  if $\Omega$ is a neighborhood of $p$ and
$$
\sup\left\{s\in \R \mid x_j(q)=O(\dsr(p,q)^s)\right\}=w_j,\quad  \,1\leq  j\leq d.
$$
This definition implies that privileged coordinates belong to the class of linearly adapted coordinates at $p$, {\it i.e.},  coordinates $(x_1,\dots, x_d)$ that satisfy
$$
\diff x_i(\Delta^{w_i}_p)\neq 0, \diff x_i(\Delta^{w_i-1}_p)=0, \quad  \,1\leq  i\leq d, 
$$
with $\Delta^0_p=\{0\}$.
Existence of privileged coordinates has been proved in \cite{Agrachev_Sarychev_filtrations,bianchini_stefani_graded,bellaiche1996TagentSpace}.

A \emph{continuously varying system of privileged coordinates} on an open $\Omega\subset M$ is a continuous map
$$
\Phi:
 (p,q) \longmapsto  \Phi_p(q)\in  \R^d
$$ 
defined on a neighborhood of the set $\{(p,p)\mid p\in \Omega\}$ in $M\times M$
such that for each $p\in \Omega$, the mapping $\Phi_p$ is a system of privileged coordinates at $p$.

The system of coordinates $\Phi$ can be used to define a pseudo-norm and a dilation, as follows:
the \emph{pseudo-norm at $p\in \Omega$} is the map  $\| \cdot \|_p:\R^d\rightarrow \R$ defined by
\begin{equation}\label{E:def_pseudo_norm}
\| (x_1,\dots,x_d) \|_p=\sum_{i=1}^d |x_i|^{1/w_i},
\end{equation}
with $(w_i)_{1\leq i\leq d}$ the weights at $p$.
Let $\lambda>0$ and $\mathfrak{d}_\lambda:\R^d\rightarrow\R^d$ be defined by 
$$
\left(
\mathfrak{d}_\lambda
\left(
y
\right)
\right)_i
=\lambda^{w_i} y_i,
\quad 
1\leq i\leq d.
$$
For $p\in\Omega$, let $\delta_\lambda^p$ be the \emph{quasi-homo\-geneous dilation centered at $p$}, 
$$
\begin{array}{cccc}
\delta_\lambda^p:& \Omega 
&\longrightarrow & \R^d
\\
&
q
&
\longmapsto
&
 \mathfrak{d}_\lambda \circ \Phi_p (q).
\end{array}
$$
Then by construction,
$$
\| \delta_\lambda^p (q) \|_p= \lambda \| \Phi_p (q) \|_p.
$$

For every  horizontal vector field  $X$  and every $p\in M$, we call \emph{nilpotent approximation of $X$ at $p$} the uniform limit $\widehat{X}$ on compact sets of $\R^d$ of the vector field $\lambda {\delta^p_\lambda}_*X $ as $\lambda\rightarrow 0^+$
  (see for instance \cite[Proposition 10.48]{ABB_nov_2016}). Given a frame $(X_1,\dots, X_m)$ of the distribution, the vector bundle $(\widehat{X}_1,\dots, \widehat{X}_m)$ endows  $\R^d $  with a structure of homogeneous space that depends on the point $p$ but neither on the frame nor the system of privileged coordinates. 
  This object is referred to as  \emph{nilpotent approximation of $(M,\Delta ,g)$ at $p$}, and is denoted by $(\R^d,(\widehat{X}_1,\dots,\widehat{X}_m))$ when referring to a specific choice of frame.
In the equiregular case, the nilpotent approximation is not only a homogeneous space but actually has a Carnot group structure 
(\cite{bellaiche1996TagentSpace}).

\subsection{Uniform distance estimates}\label{SS:Distance_Estimates}
Privileged coordinates and nilpotent approximations play a fundamental role in distance estimates which compare the pseudo norm~\eqref{E:def_pseudo_norm} with the sub-Riemannian distance.
The result below will be applied repeatedly  in the rest of the paper.
\begin{theorem}[{\cite[Theorem~2.3]{jean2014control}}]\label{T:equiregular}
Let $\bar{p}\in M$ be a regular point. There exist an open neighborhood $\Omega$ of $\bar{p}$, a continuously varying system of privileged coordinates $\Phi$ on $\Omega$, 
and two positive constants $\varepsilon $, $C $, such that 
for every pair $(p,q)\in \Omega\times \Omega$ with $\dsr(p,q)\leq \varepsilon$,  
$$
\frac{1}{C} \| \Phi_p (q) \|_p\leq \dsr(p,q) \leq C\| \Phi_p(q)\|_p.
$$
\end{theorem}

An application of this theorem is the following technical lemma (Lemma~\ref{L:our_distance_estimate}) that will be useful in later results.

In order to prove the lemma, we  introduce a useful notation for the flow of time-dependent vector fields, the so-called \emph{chronological exponential} \cite{agrachevGamkre_chron,Sachkov2004Control_theory_geometric_viewpoint}.
Let $\R\times M\ni (t,q)\mapsto X_t(q)$ be a complete time-dependent vector field, measurable and locally bounded with respect to $t$ and smooth with respect to $q$.
For $a,b\in \R$, $a\leq b$, we denote by 
$$
\chronexp
\int_a^b X_t \diff t : M\longrightarrow M
$$
the map from $M$ onto itself such that 
the curve $\gamma:[a,b]\rightarrow M $ defined by 
$
\gamma(t)=
\chronexp
\int_a^t X_\tau \diff \tau
(q_0)
$
is absolutely continuous, satisfies $\gamma(a)=q_0$ and $\dot{\gamma}(t)=X_t(\gamma(t))$ for almost every $t$.
For $a\geq b$ we set 
$
\chronexp
\int_a^b X_\tau \diff \tau
=
\left(\chronexp
\int_b^a X_{\tau} \diff \tau\right)^{-1}
$.
Let us recall the variation of constant formula. (For a reference, see Equation (2.28) in \cite{Sachkov2004Control_theory_geometric_viewpoint}; notice that here we use the standard notational rule for the composition of maps, which explains the difference between the two expressions.) If $X_\tau,Y_\tau$ are two time-dependent vector fields 
then
\begin{equation}\label{E:var_of_constants}
\chronexp \int_0^t (X_\tau +Y_\tau)\diff \tau 
=
\chronexp \int_0^t
\left(
\chronexp  \int_{\tau }^{t}
X_\sigma
\diff \sigma 
\right)_{*}
Y_\tau
\diff \tau 
\circ 
\chronexp
\int_0^t X_\tau \diff \tau,
\end{equation}
where $P_*X$ is used to denote the pushforward of the vector field $X$ along the diffeomorphism $P$.
(In order to justify the writing in \eqref{E:var_of_constants}, all the vector fields should be complete. In the following, variations formula are used for local reasonnings around a point, so that completeness can be guaranteed by multiplying all vector fields by a suitable cut-off function.)
 In particular if $X$ is a time-independent vector field, Equation~\eqref{E:var_of_constants} takes the form
\begin{equation}\label{E:var_of_c_simplified}
\chronexp \int_0^t (X +Y_\tau)\diff \tau 
=
\chronexp \int_0^t 
\e^{(\tau-t)\ad X}
Y_\tau
\diff \tau 
\circ 
\e^{t X},
\end{equation}
where for each smooth vector field $V$ on $M$ we denote by $\ad V$ the endomorphism of the space of smooth vector fields on $M$ defined by
$$
\ad V(W)=[V,W].
$$
Then $\e^{\sigma\ad X} Y_\tau$ admits the series expansion
\begin{equation}\label{E:exp_ad}
\e^{ \sigma\ad X}
Y_\tau
=
Y_\tau+\sum_{k=1}^{N-1} \frac{\sigma^k}{k!} \left[ X, \dotsc ,\left[X,Y_\tau\right]\dotsi \right]+R_N(\tau,\sigma),
\quad 
N\in \N,
\end{equation}
and there exists $C>0$ such that for all compact $K$ contained in a given coordinate neighborhood of $M$, all integer $j\geq 0$,
\begin{equation}\label{E:bound_remainder}
\left\| R_N(\tau,\sigma) \right\|_{j,K}
\leq 
\frac{C}{N!}
\e^{C \sigma \left\| X \right\|_{j+1,K} } 
\sigma^N   \left\| X  \right\|_{j+N,K}^N
\left\| Y_\tau \right\|_{j+N,K},
\end{equation}
where $\|\cdot \|_{j,K}$ denotes the semi-norm on the space of smooth vector fields
$$
\|f\|_{j,K}
=
\sup\left\{
\left|
\frac{\partial^{\alpha}f}{\partial x^\alpha}(x)
\right|
\mid x\in K,\alpha\in \N^d, |\alpha|\leq j
\right\}.
$$
(See \cite[Equation 2.24]{Sachkov2004Control_theory_geometric_viewpoint}.)
As a consequence, if $\tau\mapsto \left\| Y_\tau \right\|_{j+N,K}$ is bounded in $L^\infty([0,\tau])$, then for $\sigma$ near $0$,
$$
\left\| R_N(\tau,\sigma) \right\|_{j,K}=O(\sigma ^N).
$$

In the following, for a given frame $(X_1,\dots, X_m)$ and a given $u\in \R^m$, we denote by $X_u$ the horizontal vector field $\sum_{i=1}^m u_iX_i$.

\begin{lemma}\label{L:our_distance_estimate}
Let $(M,\Delta,g)$ be an equiregular sub-Riemannian manifold of dimension $d$, rank $m$ and step $r$.
Let $(X_1,\dotsc,X_m)$  be a frame of $\Delta$ defined on an open subset $\Omega$ of $ M$.
Pick $\bar{p}\in \Omega$, $\bar{u}\in \R^m$, and $\phi:[0,+\infty)\rightarrow[0,+\infty)$, continuous at $0$ and such that $\phi(0)=0$.

Then there exist $T>0$,  $V_{\bar{p}}\subset \Omega$, $V_{\bar{u}}\subset \R^m$  neighborhoods of $\bar{p}$, $\bar{u}$ respectively, a function $\omega:\R^+\rightarrow\R^+$ with $\omega(t)=o(t)$ at $0^+$, such that if $t\in [0,T]$, $p\in V_{\bar{p}}$, $u,v:[0,T]\rightarrow V_{\bar{u}}$ continuous at $0$ and
$$
|u-u(0)|\leq \phi , \quad |v-v(0)|\leq \phi ,\quad u(0)=v(0),
$$
then 
$$
\dsr
\left(
	\chronexp\int_0^t   X_{u(s)} \diff s (p)
	,
 	\chronexp\int_0^t   X_{v(s)} \diff s (p)
\right)
\leq
\omega(t).
$$
\end{lemma}

\begin{proof}
Without loss of generality we can assume $  u$ to be constant and  the general result follows by triangular inequality.

We apply Theorem~\ref{T:equiregular} to endow a  compact neighborhood $\Omega'\subset\Omega$ of $\bar{p}$ with a continuously varying system of privileged coordinates $\Phi$.

We fix $T>0$, $V_{\bar{p}}\subset \Omega'$ and  $V_{\bar{u}}\subset \R^m$  neighborhoods of $\bar{p}$ and $\bar{u}$, respectively, such that $\chronexp\int_0^t   X_{v(s)} \diff s (p)$ is  in $\Omega'$ for every $t\in [0,T]$, $p\in V_{\bar{p}}$, $v:[0,T]\rightarrow V_{\bar{u}}$  continuous at $0$.

Let $p\in V_{\bar{p}}$, $u \in V_{\bar{u}}$ and $v:[0,T]\rightarrow  V_{\bar{u}}$, continuous at $0$, be  such that 
$
|u-v|\leq \phi.
$
For all $t\in[0,T]$ let
$$
\gamma(t)=\e^{t X_{u } } (p) \quad\text{ and }\quad\xi(t)=\chronexp\int_0^t   X_{v(s)} \diff s (p).
$$

{\it Step 1:} rewriting $\xi$ as a perturbation of $\gamma$.
\\
Let us set $X =X_{u }$ and $Z_t=X_{v(t)}-X_u$, 
so that 
$$
\dot{\xi}(t)=X (\xi(t))+Z_t(\xi(t)).
$$
By the variation  of constants formula,
$$
\xi(t)
=
\chronexp \int_0^t(X +Z_s)\diff s(p)
=\chronexp \int_0^t 
\e^{(s-t)\ad X}
Z_s
\diff s
\left(
\e^{t X}(p)\right).
$$
Let $W^{\new{t}}_s=\e^{(s-t)\ad X}Z_s$ and denote its integral curve by
$$
\eta_t(\tau)=\chronexp \int_0^\tau W^{t}_s\diff s (\gamma(t)),\qquad \forall \tau\in (0,t).
$$
 Hence the problem consists in proving that the distance
$
\dsr
\left(
\xi(t)
,
\eta_t(0)
\right)
$
is a $o(t)$. 

Let us first establish a broader bound on 
$
\dsr
\left(
\xi(t)
,
\eta_t(\tau)
\right)
$.
 For every $\tau\in (0,t)$, by applying the variation of constants formula at $\xi(\tau)$ we get
\begin{align*}
\xi(t)
&=
\chronexp \int_\tau^t(X +Z_s)\diff s(\xi(\tau))
=\chronexp \int_\tau^t 
\e^{ (s-t)\ad X}
Z_s
\diff s
\left(
\e^{(t-\tau) X}(\xi(\tau))\right)\\
&=\chronexp \int_\tau^t 
W^t_s\diff s
\left(
\e^{(t-\tau) X}(\xi(\tau))\right).
\end{align*}
On the other hand,
$$
\xi(t)=\eta_t(t)
=
\chronexp \int_\tau^t 
W^t_s
\diff s
\left(
\eta_t(\tau)\right),
$$
and therefore $\eta_t(\tau)=\e^{(t-\tau) X}(\xi(\tau))$ for all $\tau\in [0,t]$. In particular there exists $C>0$ such that
\begin{equation}\label{E:dist_eta_t}
\dsr (\xi(t),\eta_t(\tau))\leq C(t-\tau).
\end{equation}
(See Figure \ref{F:eta_t}.)
\begin{figure}[ht!]
\begin{center}
\begingroup%
    \setlength{\unitlength}{10cm}%
  \begin{picture}(1,0.6)%
    \put(0,0){\includegraphics[width=\unitlength,page=1]{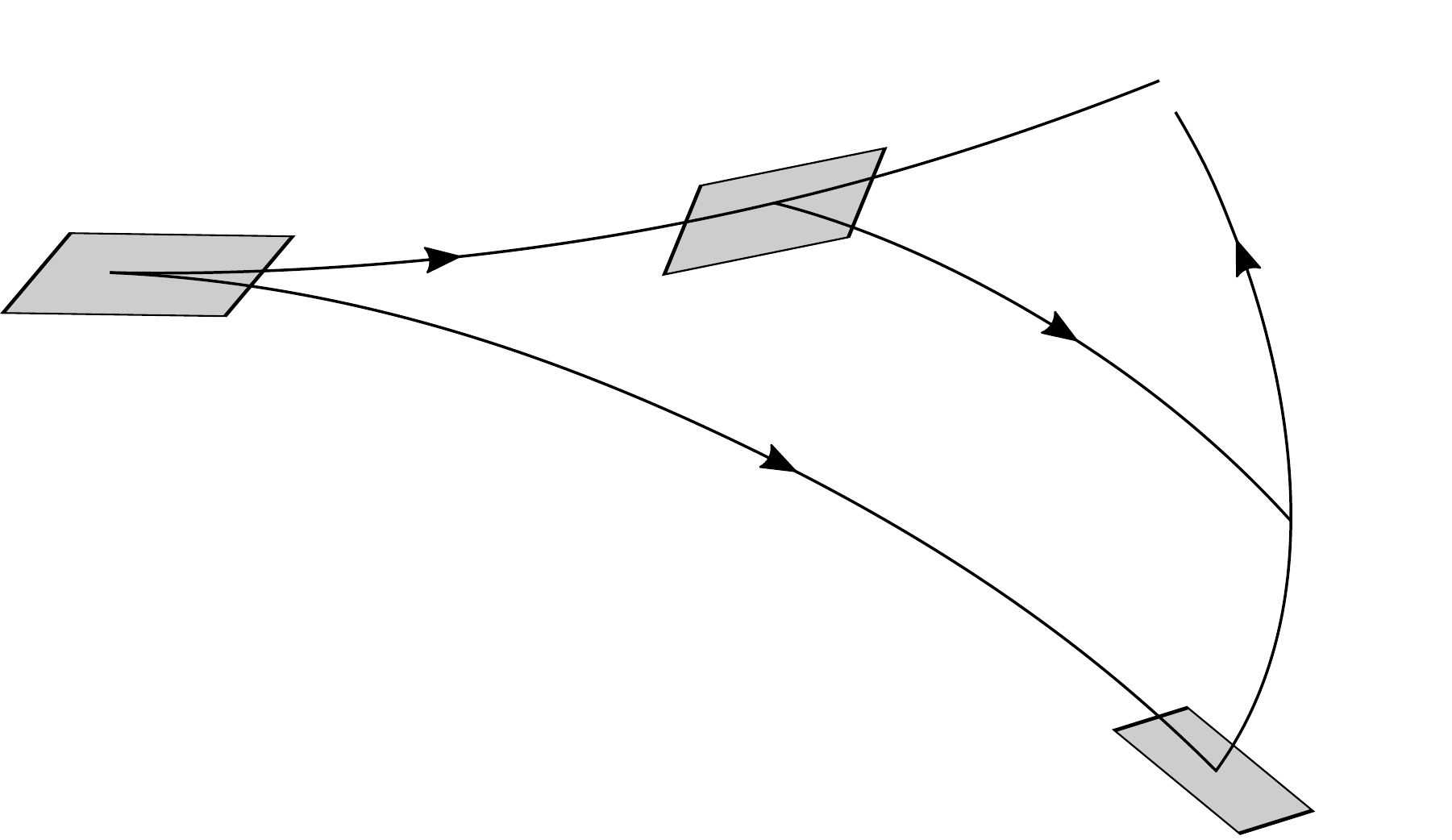}}%
    \put(0,0){\includegraphics[width=\unitlength,page=2]{dessin_revised.pdf}}%
    \put(0.87,0.06){$\gamma(t)$}%
    \put(0.06 ,0.43){ $p$ }%
    \put(0.74,0.55){$\xi(t)$}
    \put(0.24,0.42){ $X+Z_s$ }%
    \put(0.69 ,0.3){$X$}%
    \put(0.5,0.48){$\xi(\tau)$}%
    \put(0.88,0.4){$W^{t}_s$}%
    \put(0.90 ,0.21 ){$\eta_t(\tau)$}%
    \put(0.5,0.21){$X$}%
  \end{picture}%
\endgroup%
\caption{Representation of the curve $\eta_t$, where $\eta_t(\tau)$ can be seen both as the evaluation at time $\tau$ of an integral curve of the non-horizontal vector field $W^{t}_s$ and as the endpoint of the concatenation of an integral curve of $X+Z_s$ over $[0,\tau]$ and an integral curve of $X$ of duration $t-\tau$.\label{F:eta_t}}
\end{center}
\end{figure}

\noindent{\it Step 2:} bounding the pseudo-norm  centered at $\xi(t)$ of  $\gamma(t)$.

By possibly reducing $T$, we can assume that $\eta_t(\tau)\in \Omega'$ for every $0\leq \tau\leq t\leq T$.
 We then use the privileged coordinates $\Phi:\Omega\times \Omega\rightarrow \R^d$ at $\xi(t)$ to compute the pseudo-norm $\| \Phi_{\xi(t)}(\eta_t(\tau))\|_{\xi(t)}$.
 
Denote by  $(n_1,\dotsc,n_r)$  the growth vector of the sub-Riemannian structure, by $(w_1,\dotsc,w_d)$ the corresponding weights, and by $x=(x_1,\dotsc,x_d)$ the coordinates $\Phi_{\xi(t)}$.
We want  to evaluate for all $1\leq i\leq d$  the absolute value of
$$
x_i(\gamma(t))=-\int_0^t \left( W_s^t(\eta_t(s)) \right)_i \diff s.
$$

Taking $N=r$ in the expansion~\eqref{E:exp_ad} we have
\begin{equation}\label{E:expansion_Wst}
W_s^t =
Z_s+\sum_{k=1}^{r-1} \frac{(s-t)^k}{k!} \left[ X, \dotsc ,\left[X,Z_s\right]\dotsi \right]+R_r(s,s-t).
\end{equation}
The remainder $R_r$ can be bounded using \eqref{E:bound_remainder}. For each coordinate $(R_r(s,s-t))_i$, $1\leq i\leq d$, we have that
$$
\left\| (R_r(s,s-t))_i \right\|_{0,\Omega'}
\leq 
\frac{C}{r!}
\e^{C (s-t) \left\| X \right\|_{1,\Omega'} } 
(s-t)^r   \left\| X  \right\|_{r,\Omega'}^r
\left\| Z_s \right\|_{r,\Omega'}.$$
Since   $|v-u|\leq \phi$,
the compactness of $\Omega'$ and $V_{\bar{u}}$ implies the existence of $C$ uniform such that 
$$
  | (R_r(s,s-t))_i  | \leq  C (s-t)^{r}\phi(s).
$$

We now use non-holonomic order arguments (see \cite[Section 2.1]{jean2014control}) to bound the other terms in the expansion in \eqref{E:expansion_Wst}.

Let  $V^k_s=(\ad X)^{k-1}Z_s$ for $k\geq 1$. The vector fields $X$ and $Z_s$ being horizontal, $V^k_s\in \Delta^{k}$. 
As a consequence, the vector $V^k_s$ has a non-holonomic order greater than or equal to $-k$ at any $p\in \Omega'$. Coordinate-wise, at $\xi(t)$, $\mathrm{ord}_{\xi(t)}(\partial_{x_i})=-w_i$, so that 
 $(V^k_s)_i$ then has non-holonomic order 
 $$
 \mathrm{ord}_{\xi(t)}((V^k_s)_i)\geq \max(w_i-k, 0).
 $$
Since $ V^k_s$ depends linearly on $v-u$, there exists $C_k>0$ such that 
 for $q\in \Omega'$ sufficiently close to $\xi(t)$,
$$
\left| (V^k_s)_i(q)\right|
\leq 
C_k\phi(s)\,\dsr (\xi(t),q)^{\max(w_i-k, 0)}.
$$
Using estimate \eqref{E:dist_eta_t}, 
$$
\left| (V_s^k)_i(\eta_t(s))\right|
\leq 
C_k\phi(s) t^{\max(w_i-k, 0)}.
$$
Thus
$$
\int_0^t  \left|\left( Z_{s}\right)_i(\eta_t(s))\right|
   \diff s  \leq C_1 t^{\max(w_i-1, 0)} \int_0^t \phi(s) \diff s=t^{w_i}\psi^1(t),
$$
and 
$$
\int_0^t  
\left|
\left(
	(t-s)^{k-1} V_s^k
\right)_i(\eta_t(s))\right| \diff s  
\leq C_k
t^{\max(w_i-k, 0)} \int_0^t  
 (t-s)^{k-1} \phi(s)  \diff s=t^{w_i}\psi^k(t),
$$
where for all positive integer $k$ we denote by $\psi^k$ the positive bounded function $\psi^k:[0,T]\ni  t\mapsto C_k t^{-k}\int_0^t (t-s)^{k-1}\phi(s)\diff s$, which is continuous and such that $\psi^k(0)=0$ (by continuity at $0$ of $\phi$).
Thus, for all $1\leq i\leq d$,
$$
|x_i(\gamma(t))|\leq t^{w_i} \Psi_i(t), 
$$
with $\Psi_i:[0,T]\rightarrow\R^+$ a bounded function  continuous at $0$ such that $\Psi_i(0)=0$. Hence we have the uniform bound
\begin{equation}\label{E:Real_estimate}
\|\Phi_{\xi(t)}(\gamma(t))\|_{\xi(t)}\leq \sum_{i=1}^d t (\Psi_i(t))^{1/w_i}=\psi(t)\,t,
\end{equation}
with $\psi:[0,T]\rightarrow\R^+$ a  function continuous at $0$ such that $\psi(0)=0$.

\noindent {\it Step 3:} Uniform estimates.

Let $\varepsilon,C>0$ be the constants associated with the neighborhood $\Omega$ in Theorem~\ref{T:equiregular}. Take $T>0$ such that if $0\leq t\leq T$, 
$$
\dsr \left( \gamma(t), \xi(t) \right)\leq  \varepsilon.
$$
Therefore, by Theorem~\ref{T:equiregular},
$$
\dsr \left( \gamma(t),  \xi(t)  \right)\leq C\|\Phi_{\xi(t)}(\gamma(t))\|_{\xi(t)},
$$
and, plugging in \eqref{E:Real_estimate}, we get
$$
\dsr \left( \gamma(t),  \xi(t)  \right)\leq C \psi(t)\,t .
$$
Thus, letting $\omega(t)=C  \psi(t)t $, we have the uniform bound
$$
\dsr
\left(
	\e^{tX_{u}}  (p)
	,
 	\chronexp\int_0^t   X_{v(s)} \diff s (p)
\right)
\leq
\omega(t).
$$
\end{proof}

\section{$C^1_H$-Whitney condition on sub-Riemannian manifolds}\label{S:Whitney_cond}

We begin this section by proposing a definition of $C_H^1$-Whitney condition for curves in sub-Riemannian structures that requires  the choice of a  frame of the sub-Riemannian structure. We then show that such a definition is  intrinsic  and we explore a few consequences.

Recall that for a given frame $(X_1,\dots, X_m)$ and a given $u\in \R^m$, we denote by $X_u$ the horizontal vector field $\sum_{i=1}^m u_iX_i$.
\subsection{Whitney frame-wise condition}\label{SS:Equireg_Whitney_condition}

We denote by $(M,\Delta,g)$   a sub-Riemannian manifold of dimension $d$. 
Let $(X_1,\dotsc,X_m)$ be  a frame of $\Delta$ defined on an open subset $\Omega$ of $ M$.

We say that $(f,L)$ satisfies  property $\mathcal{P}_X$ on $K$ if the following is true.

\noindent
\begin{minipage}[t]{\widthof{$\mathcal{P}_X$ :}}
\noindent
$\mathcal{P}_X$ :
\end{minipage}
\begin{minipage}[t]{\linewidth-\widthof{ $\mathcal{P}_X$ :}- \parindent}
\noindent
$K$ is a compact in $\R$, $f:K\rightarrow \Omega$ and $L:K\rightarrow TM$ 
are continuous. Moreover, there exist 
$u:K\to \R^m$ continuous and $\omega:\R^+\rightarrow\R^+$ such that 
$L(t)=X_{u(t)}(f(t))$ for all $t\in K$, $\omega(s)=o(s)$ at $0^+$, and 
$$
\dsr\left(f(t),\e^{(t-s)X_{u(s)}}f(s)	\right)\leq \omega(|t-s|)\qquad\mbox{for all $t,s\in K$}.
$$
\end{minipage}

This definition is motivated by the following proposition.

\begin{proposition}\label{P:Taylor_exp}
Let $(M,\Delta,g)$ be a sub-Riemannian manifold and $(X_1,\dotsc,X_m)$ be  a frame of $\Delta$ defined on an open subset $\Omega$ of $ M$.
Let $f:\R\rightarrow \Omega$ be a $C^1_H$ curve. Then $(\left.f\right|_{K},\dot{\left.f\right|}_{K})$ satisfies  $\mathcal{P}_X$ on $K$ for any compact subset $K$ of $\R$.
\end{proposition}

\begin{proof}
 Let $u\in C(\R,\R^m)$ be such that $X_{u(\tau)}(f(\tau))=\dot{f}(\tau)$ for every $\tau\in \R$. Then
$$
f(t)=\chronexp \int_s^t X_{u(\tau)} \diff \tau (f(s))
$$
for every $s,t\in \R$.

Let $K$ be a compact subset of $\R$ and $K'$ a connected compact set of $\R$ containing $K$. For every $s\in [0,+\infty)$, set  $\phi(s)=\sup_{t\in K'} |u(t+s)-u(t)|$. Notice that $\lim_{s\rightarrow 0^+}\phi(s)=0$ by uniform continuity of $u$ on compact subsets of $\R$.

Let us first consider the equiregular case. Let $s\in K$.
We set $v=u(s)\in \R^m$ and we apply Lemma~\ref{L:our_distance_estimate} at $f(s)$ with $\phi$ as above. 
We get that there exist a neighborhood $V_s\subset \R$ of $s$ and a function $\omega_s:\R^+\rightarrow\R^+$ such that $\omega_s(t)=o(t)$ at $0^+$ and
\begin{equation}\label{E:Taylor_ponct}
\dsr\left(f(t),\e^{(t-t')  X_{u(t')}}f(t')	\right)\leq \omega_s(|t-t'|)
\end{equation}
for all $t,t'\in V_s\cap K$. Taking a finite cover $V_{s_1},\dots,V_{s_N}$ of $K$ we have that $\omega(t)=\max_{1\leq i\leq N} \omega_{s_i}(t)=o(t)$ at $0^+$ and we deduce that $(\left.f\right|_{K},\dot{\left.f\right|}_{K})$ satisfies  $\mathcal{P}_X$ on $K$.

Assume now that the manifold is singular. For every $s\in K$ there exists a neighborhood $\Omega_s$ of $f(s)$ contained in $\Omega$ and a desingularization $\psi:\widetilde{\Omega}_s \rightarrow \Omega_s $ such that $\widetilde{\Omega}_s$ is an open set in an equiregular sub-Riemannian manifold $\widetilde{M}$. Let $(\widetilde{X}_1,\dots,\widetilde{X}_m)$ be the lifted frame of $(X_1,\dots ,X_m)$ on $\widetilde{\Omega}_s$.

Then we fix $\widetilde{f}(s)\in \psi^{-1}(f(s))$ and we set for all $t$ in a neighborhood of $s$
$$
\widetilde{f}(t)=\chronexp \int_s^t \widetilde{X}_{u(\tau)} \diff \tau (\widetilde{f}(s)).
$$
By construction, $\psi(\widetilde{f})=f$. 
We apply the equiregular reasoning on $\widetilde{f}$ at $s$, and we get the existence of 
 $\omega_s:\R^+\rightarrow\R^+$ such that $\omega_s(t)=o(t)$ at $0^+$ and
$$
\dsrlift \left(\widetilde{f}(t),\e^{(t-t')\widetilde{X}_{u(t')}}\widetilde{f}(t')	\right)\leq \omega_s(|t-t'|)
$$
for all $t$ and $t'$ close enough to $s$. By projecting this inequality (see Inequality~\eqref{E:proj_lift_dist}), we get 
$$
\dsr \left( f(t),\e^{(t-t') X_{u(t')}} f(t')	\right)\leq \omega_s(|t-t'|).
$$
We conclude with the same compactness argument as in the equiregular case.
\end{proof}

We then define the $C^1_H$-Whitney condition and the $C^1_H$ extension property as follows.

\begin{definition}[$C^1_H$-Whitney condition]\label{D:C1H_Whitney_cond}
Let $K$ be a closed subset of $\R$,
$f:K\rightarrow M$ continuous, and $L:K\rightarrow TM$ continuous.
We say that the \emph{$\mathcal{C}^1_H$-Whitney condition holds for $(f,L)$ on $K$} if for every $t\in K$, there exist a compact neighborhood $K'$ of $t$ in $K$, an open set $\Omega\subset M$ and a local frame $X$ of $\Delta$ on $\Omega$ such that $(\left.f\right|_{K'},\left. L\right|_{K'})$ satisfies $\mathcal{P}_X$ on $K'$.
\end{definition}

\begin{definition}[$C^1_H$ extension property]\label{D:C1H_extension_prop}
We say that a sub-Riemannian manifold $M$ has the \emph{$C^1_H$ extension property} if for all closed subset $K$ of $\R$, all pair $(f,L):K\rightarrow M\times TM$ continuous satisfying the $C^1_H$-Whitney condition, there exists a $C^1_H$ curve $\gamma:\R\rightarrow M$ such that 
$$
\gamma_{ |K}=f, 
\quad
\dot{\gamma}_{ |K}=L.
$$
\end{definition}

In the case where $(M,\Delta,g)$ is equiregular, we are able to show that the property of satisfying $\mathcal{P}_X$ is intrinsic to the curve and does not depend on the choice of the frame.

\begin{proposition}\label{P:real_Px<->Py}
Assume  $(M,\Delta,g)$ to be an equiregular sub-Riemannian manifold.
Let $(X_1,\dotsc,X_m)$ and $(Y_1,\dotsc,Y_m)$ be two frames of $\Delta$ defined on an open subset $\Omega$ of $ M$. Let $K\subset \R$ be compact and $(f,L):K\rightarrow \Omega\times TM$ be continuous.
Then $(f,L)$ satisfies $\mathcal{P}_X$ on $K$ if and only if it satisfies $\mathcal{P}_Y $ on $K$.
\end{proposition}
\begin{proof}
Let us assume that $(f,L)$ satisfies $\mathcal{P}_X$ on $K$. In particular,
there exists $\omega_X:\R^+\rightarrow\R^+$, such that $\omega_X(t)=o(t)$ at $0^+$ and, for all $t,s\in K$,
$$
\dsr\left(f(s),\e^{(s-t)X_{u(t)}}f(t)	\right)\leq \omega_X(|s-t|),
$$
with $u:K\to \R^m$ continuous such that $L(t)=X_{u(t)}(f(t))$ for every $t\in K$.
Let us prove that $(f,L)$ satisfies $\mathcal{P}_Y$.

Let $v:K\rightarrow \R^m$ be a continuous map such that $Y_{v(t)}(f(t))=L(t)$ for every $t\in K$.
Since $(X_1,\dots,X_m)$ and $(Y_1,\dots,Y_m)$ are both frames of $\Delta$, there exist smooth functions $(c_{ij})_{1\leq i,j\leq m}$ such that for all $q\in \Omega$,
$$
Y_j(q)=\sum_{i=1}^m c_{ij}(q)X_i(q).
$$
Then
$$
\sum_{j=1}^m v_j(t) Y_j(q)=\sum_{i=1}^m \left( \sum_{j=1}^m v_j(t) c_{ij}(q)\right) X_i(q),\quad t\in K, \;q\in M.
$$

As a consequence of  Lemma~\ref{L:our_distance_estimate}, for all $t\in K$  there exist $T_t>0$,  $V_{f(t)}\subset \Omega$, $V_{u(t)}\subset \R^m$ neighborhoods of $f(t)$, $u(t)$   respectively, there exists $\omega_t:\R^+\rightarrow\R^+$ with $\omega_t(s)=o(s)$ at $0^+$, such that if $p\in V_{f(t)}$, $u\in V_{u(t)}$ and $v\in\R^m$ satisfy
$X_u(p)=Y_v(p )$
then 
$$
\dsr\left(\e^{s X_u}p,\e^{sY_v}p\right)\leq \omega_t(s) ,\quad  s\in [0,T_t].
$$

By compactness of $f(K)$, 
there exists a finite cover $V_{f(t_1)},\dots,V_{f(t_N)}$ of $f(K)$.
Then $\omega(t)=\max_{1\leq i\leq N} \omega_{t_i}(s)=o(s)$ at $0^+$ and we deduce that 
$$
\dsr\left(\e^{(s-t) X_{u(t)}}f(t),\e^{(s-t)Y_{v(t)}}f(t)\right)\leq \omega(|s-t|)
$$
for all $s,t\in K$ close enough.
Then, for $s,t\in K$ close enough,
$$
\begin{aligned}
\dsr\left(f(s),\e^{(s-t)Y_{v(t)}}f(t)	\right)
&\leq \dsr\left(f(s),\e^{(s-t)X_{u(t)}}f(t)	\right)+
\dsr\left( \e^{(s-t)X_{u(t)}}f(t)	,\e^{(s-t)Y_{v(t)}}f(t)\right)
\\
&\leq \omega_X(|s-t|)+
 \omega (|s-t|).
\end{aligned}
$$
Thus $(f,L)$ satisfies $\mathcal{P}_Y$ on $K$.
\end{proof}

\begin{remark}\label{continuous-selection}
The proof also shows that 
the definition of $\mathcal{P}_X$ does not depend on the choice of the continuous function $u$ such that 
$L(t)=X_{u(t)}(f(t))$ for every $t\in K$.
\end{remark}

\subsection{Forward and backward Whitney condition}\label{SS:Forward_Backward_WC}
To extend the study of the $\mathcal{C}^1_H$-Whitney condition to singular sub-Riemannian manifolds, we first have to break the symmetry in the definition of $\mathcal{P}_X$ by comparing only flows going forward or backward in time
(that is, by requiring either $s<t$ or $s>t$ in the statement of $\mathcal{P}_X$).
This new definition has two virtues. First, the asymmetric definition turns out to be equivalent to the symmetric one and it is easier to lift on a desingularized manifold. Second, the asymmetric definition lends itself well to the use of dilations, which will be useful in Section~\ref{S:Strong_pliability}.

Consider an equiregular sub-Riemannian manifold  $(M,\Delta,g)$ of dimension $d$, rank $m$ and step $r$. Let $(X_1,\dotsc,X_m)$   be a frame of $\Delta$ defined on an open subset $\Omega$ of $ M$.

We say that $(f,L)$ satisfies the property $\mathcal{P}_X$-forward, denoted by $\mathcal{P}_X^F$, or $\mathcal{P}_X$-backward, denoted by $\mathcal{P}_X^B$, on $K$ if the following is true.

\noindent
\begin{minipage}[t]{\widthof{$\mathcal{P}_X^F$ :}}
\noindent
$\mathcal{P}_X^F$ :
\end{minipage}
\begin{minipage}[t]{\linewidth-\widthof{ $\mathcal{P}_X$ :}- \parindent}
\noindent
$K$ is a compact subset of $\R$, $f:K\rightarrow \Omega$ is continuous, and $L:K\rightarrow TM$ is such that 
$L(t)=X_{u(t)}(f(t))$ for all $t\in K$ for some $u:K\to \R^m$ continuous. Moreover
there exists $\omega:\R^+\rightarrow\R^+$, such that $\omega(t)=o(t)$ at $0^+$ and
$$
\dsr\left(f(t),\e^{(t-s)X_{u(s)}}f(s)	\right)\leq \omega(t-s) 
\qquad
\forall t>s\in K.
$$
\end{minipage}

\noindent
\begin{minipage}[t]{\widthof{$\mathcal{P}_X^B$ :}}
\noindent
$\mathcal{P}_X^B$ :
\end{minipage}
\begin{minipage}[t]{\linewidth-\widthof{ $\mathcal{P}_X$ :}- \parindent}
\noindent
$K$ is a compact subset of $\R$, $f:K\rightarrow \Omega$ is continuous, and $L:K\rightarrow TM$  
 is such that 
$L(t)=X_{u(t)}(f(t))$ for all $t\in K$ for some $u:K\to \R^m$ continuous. 
 Moreover
there exists $\omega:\R^+\rightarrow\R^+$, such that $\omega(t)=o(t)$ at $0^+$ and
$$
\dsr\left(f(t),\e^{(t-s)X_{u(s)}}f(s)	\right)\leq \omega(s-t) 
\qquad
\forall t<s\in K.
$$
\end{minipage}

Again, we emphasize that the difference between $\mathcal{P}_X^B$ and $\mathcal{P}_X^F$ is in the requirement that either $t>s$ (for $\mathcal{P}_X^F$) or $s<t$ (for $\mathcal{P}_X^B$). 
In analogy with Definition~\ref{D:C1H_Whitney_cond}, we introduce the following notion.

\begin{definition}[Backward and forward $C^1_H$-Whitney condition]
Let $K$ be a closed subset of $\R$, $f:K\rightarrow M$ continuous, and $L:K\rightarrow TM$ continuous be such that $L(t)\in \Delta_{f(t)}$ for all $t\in K$. We say that the \emph{backward  (respectively, forward) $\mathcal{C}^1_H$-Whitney condition holds for $(f,L)$ on $K$} if for every $t\in K$ there exist a compact neighborhood $K'$ of $t$ in $K$, an open set $\Omega\subset M$ and a local frame $X$ of $\Delta$ on $\Omega$ such that $(\left.f\right|_{K'},\left. L\right|_{K'})$ satisfies $\mathcal{P}_X^B$ (respectively, $\mathcal{P}_X^F$).
\end{definition}

The reasoning in Section \ref{SS:Equireg_Whitney_condition} still holds when we consider $\mathcal{P}_X$-backward and $\mathcal{P}_X$-forward. Hence the following result.
\begin{proposition}\label{P:PXB<=>PYB}
Assume  $(M,\Delta,g)$ to be an equiregular sub-Riemannian manifold.
Let $(X_1,\dotsc,X_m)$ and $(Y_1,\dotsc,Y_m)$ be two frames of $\Delta$ defined on an open subset $\Omega$ of $ M$.
Let $K\subset \R$ be compact and $(f,L):K\rightarrow \Omega\times TM$ be continuous.
Then $(f,L)$ satisfies $\mathcal{P}_X^F$ (respectively, $\mathcal{P}_X^B$) on $K$ if and only if it satisfies $\mathcal{P}_Y^F $  (respectively, $\mathcal{P}_Y^B$)  on $K$.
\end{proposition}

Proposition~\ref{P:lim_dilation_forward} below reformulates the forward  $C^1_H$-Whitney condition using dilations in privileged coordinates.

\begin{proposition}\label{P:lim_dilation_forward}
Let $(M,\Delta,g)$ be an equiregular sub-Riemannian manifold. 
Let $K\subset \R$ be compact set, $\Omega$ be an open subset of $M$ and $(f,L):K\rightarrow \Omega\times TM$ be continuous. Let $u:K\to \R^m$ be continuous such that $L(t)=X_{u(t)}(f(t))$ for every $t\in K$.
Assume that there exists a continuously varying system of privileged coordinates $\Phi:\Omega\times \Omega \rightarrow \R^d$ as in Theorem~\ref{T:equiregular}. Then
the pair $(f,L)$ satisfies $\mathcal{P}_X^F$  if and only if for all $l\in K$, for all sequences $(a_n)_{n\in \N}$, $(b_n)_{n\in \N}$ in $K$ such that $a_n<b_n$ and $a_n,b_n\rightarrow l\in K$, we have 
$$
\lim_{n\to \infty} \delta^{f(b_n)}_{\frac{1}{b_n-a_n}}(f(a_n))= \e^{-\widehat{X}_{u(l)}}(0).
$$
\end{proposition}

\begin{proof}
Let $l\in K$, $(a_n)_{n\in \N}$, $(b_n)_{n\in \N}$ in $K$ be such that $a_n<b_n$ and $a_n,b_n\rightarrow l\in K$. By  assumption there exists two positive  constants $\varepsilon $, $C $, such that 
for every pair $(p,q)\in \Omega\times \Omega$ with $\dsr(p,q)\leq \varepsilon $, it holds 
$$
\frac{1}{C } \| \Phi_p (q) \|_p\leq \dsr(p,q) \leq C\| \Phi_p(q)\|_p.
$$
Then, for $n$ large enough,
\begin{multline}\label{E:inequality_forward}
\frac{1}{C}
\left\|
\Phi_{f(b_n)} \left( \e^{(b_n-a_n) X_{u(a_n)} }f(a_n)\right)
\right\|_{f(b_n)}
\leq
\\
\dsr\left(
f(b_n)
,
\e^{(b_n-a_n) X_{u(a_n)}}  f(a_n)
\right)
\\
\leq 
C
\left\|
\Phi_{f(b_n)} \left( \e^{(b_n-a_n) X_{u(a_n)} }f(a_n)\right)
\right\|_{f(b_n)}.
\end{multline}
By introducing a dilation in the pseudo norm, we get
$$
\begin{aligned}
\left\|
\Phi_{f(b_n)} \left( \e^{(b_n-a_n) X_{u(a_n)} }f(a_n)\right)
\right\|_{f(b_n)}
&= 
(b_n-a_n)\left\| \mathfrak{d}_{\frac{1}{b_n-a_n}}\circ\Phi_{f(b_n)} \left( \e^{(b_n-a_n) X_{u(a_n)} }f(a_n) \right) \right\|_{f(b_n)}
\\
&
=
(b_n-a_n)\left\|
\delta_{\frac{1}{b_n-a_n}}^{f(b_n)}\left(\e^{(b_n-a_n) X_{u(a_n)}}f(a_n) \right)
\right\|_{f(b_n)}.
\end{aligned}
$$
Denoting $t_n=b_n-a_n$, we get
$$
\delta_{1/t_n}^{f(b_n)}\left(\e^{t_nX_{u(a_n)}}\left(f(a_n)\right)\right)
=
\e^{t_n {\delta_{1/t_n}^{f(b_n)}}_*X_{u(a_n)}}\left(\delta_{ 1/t_n}^{f(b_n)}\left(f(a_n)\right)\right).
$$
Hence 
\begin{equation}\label{E:dilation_flow_pseudo_norm}
\delta_{ 1/t_n}^{f(b_n)}\left(f(a_n)\right)=
\e^{-t_n {\delta_{1/t_n}^{f(b_n)}}_*X_{u(a_n)}}\left(   \delta_{1/t_n}^{f(b_n)}\left(\e^{t_nX_{u(a_n)}}\left(f(a_n)\right)\right)\right).
\end{equation}
Since $u$ and $f$ are continuous on $K$ and $a_n,b_n\rightarrow l\in K$, $t_n {\delta_{1/t_n}^{f(b_n)}}_* X_{u(a_n)}$ locally uniformly converges towards $\widehat{X}_{u(l)}$.
This is a consequence of the local uniform convergence of $\lambda \delta_{1/\lambda*}X_u$ towards $\widehat{X}_u$ as $\lambda\rightarrow 0$ (see \cite[Proposition 10.48]{ABB_nov_2016}).
Thus 
$
\e^{-t_n {\delta_{1/t_n}^{f(b_n)}}_*X_{u(a_n)}}
$
locally uniformly converges towards $e^{- \widehat{X}_{u(l)}} $.

If $(f,L)$ satisfies $\mathcal{P}_X^F$ then Equation~\eqref{E:inequality_forward} implies that 
$$
\lim_{n\to \infty} \left\| \delta_{1/t_n }^{f(b_n)}\left(\e^{t_nX_{u(a_n)}}f(a_n) \right)
\right\|_{f(b_n)}
= 0.
$$
It follows from \eqref{E:dilation_flow_pseudo_norm} and the local uniform convergence of $\e^{-t_n {\delta_{1/t_n}^{f(b_n)}}_*X_{u(a_n)}} $ towards $e^{- \widehat{X}_{u(l)}} $ that
$$
\lim_{n\to \infty}
\delta^{f(b_n)}_{ 1/t_n }(f(a_n))
=  \e^{-\widehat{X}_{u(l)}}(0).
$$

Conversely, assume now that for all $l\in K$, for all sequences $(a_n)_{n\in \N}$, $(b_n)_{n\in \N}$ in $K$ such that $a_n<b_n$ and $a_n,b_n\rightarrow l\in K$, we have
$$
\lim_{n\to \infty} \delta^{f(b_n)}_{\frac{1}{b_n-a_n}}(f(a_n))= \e^{-\widehat{X}_{u(l)}}(0).
$$

To prove that $(f,L)$ satisfies $\mathcal{P}_X^F$, we prove that 
$$
\lim_{t\rightarrow 0}\sup_{\substack{a,b\in K\\0<b-a<t}}
\frac{1}{b-a}
\dsr\left(
f(b)
,
\e^{(b-a) X_{u(a)}}  f(a)
\right)=0.
$$
Thanks to estimate \eqref{E:inequality_forward}, we are left to prove that
$$
\lim_{t\rightarrow 0}\sup_{\substack{a,b\in K\\0<b-a<t}}
\left\|
\e^{(b-a) {\delta_{1/(b-a)}^{f(b)}}_*X_{u(a)}}\left(\delta_{1/(b-a)}^{f(b)}\left(f(a)\right)\right)
\right\|_{f(b)}=0.
$$

Assume that there exist $\eta>0$,  $(a_n)_{n\in\N}$  and $(b_n)_{n\in\N}$ in $K$ such that, for all $n\in \N$, $0<b_n - a_n<1/n$ and 
$$
\left\|
\e^{(b_n-a_n) {\delta_{1/(b_n-a_n)}^{f(b_n)}}_*X_{u(a_n)}}\left(\delta_{1/(b_n-a_n)}^{f(b_n)}\left(f(a_n)\right)\right)
\right\|_{f(b_n)}>\eta.
$$
Up to extraction, the sequences $(a_n)_{n\in\N}$ and $(b_n)_{n\in\N}$  converge to some $l\in K$, so that 
$$
 \delta^{f(b_n)}_{\frac{1}{b_n-a_n}}(f(a_n)) \rightarrow \e^{-\widehat{X}_{u(l)}}(0).
$$
Moreover
$$
\e^{(b_n-a_n) {\delta_{1/(b_n-a_n)}^{f(b_n)}}_*X_{u(a)}} \rightarrow\e^{\widehat{X}_{u(l)}}
$$
locally uniformly on $\R^d$. Thus
$$
\e^{(b_n-a_n) {\delta_{1/(b_n-a_n)}^{f(b_n)}}_*X_{u(a_n)}}\left(\delta_{1/(b_n-a_n)}^{f(b_n)}\left(f(a_n)\right)\right)\longrightarrow 0
$$
in $\R^d$, concluding the contradiction argument.
\end{proof}

With an analogous proof we obtain the similar backward result.
\begin{proposition}\label{P:lim_dilation_backward}
Let $(M,\Delta,g)$ be an equiregular sub-Riemannian manifold. Let $K\subset \R$ be compact and $(f,L):K\rightarrow \Omega\times TM$ be continuous.
Let $u:K\to \R^m$ be continuous such that $L(t)=X_{u(t)}(f(t))$ for every $t\in K$.
Assume that there exists a continuously varying system of privileged coordinates $\Phi:\Omega\times \Omega \rightarrow \R^d$ as in Theorem~\ref{T:equiregular}.
Then the pair $(f,L)$  satisfies $\mathcal{P}_X^B$ on a compact $K\subset \R$ if and only if for all $l\in K$, for all sequences $(a_n)_{n\in \N}$, $(b_n)_{n\in \N}$ in $K$ such that $a_n<b_n$ and $a_n,b_n\rightarrow l\in K$, we have that 
$$
\lim_{n\to \infty} \delta^{f(a_n)}_{\frac{1}{b_n-a_n}}(f(b_n))= \e^{\widehat{X}_{u(l)}}(0).
$$
\end{proposition}

We are now ready to prove that forward and backward $C^1_H$-Whitney conditions are equivalent.

\begin{proposition}\label{P:Wg<=>Wd}
Assume  $(M,\Delta,g)$ to be an equiregular sub-Riemannian manifold.
Let $(X_1,\dotsc,X_m)$ be a frame of $\Delta$ defined on an open subset $\Omega$ of $ M$.
Let $K\subset \R$ be compact and $(f,L):K\rightarrow \Omega\times TM$ be continuous.
The pair $(f,L)$  satisfies $\mathcal{P}_X^F$ on $K$ if and only if $(f,L)$ satisfies $\mathcal{P}_X^B $ on $K$.
Both are equivalent to $(f,L)$ satisfying $\mathcal{P}_X $ on $K$.
\end{proposition}

\begin{proof}
By symmetry of the definitions, we only   prove that $\mathcal{P}_X^F \Rightarrow \mathcal{P}_X^B $ using Proposition~\ref{P:lim_dilation_backward}. The converse would use Proposition~\ref{P:lim_dilation_forward}.

By applying Theorem~\ref{T:equiregular} at $f(l)$, we select a compact neighborhood $\Omega$ of $f(l)$, a continuously varying system of privileged coordinates $\Phi$ on $\Omega$ and two constants $\varepsilon$ and $C$ such that for all $(p,q)\in \Omega\times \Omega$ with $\dsr(p,q)\leq \varepsilon$,
$$
\frac{1}{C}\left\|\Phi_p(q)\right\|_{p}
\leq 
\dsr(p,q)
\leq 
C\left\|\Phi_p(q)\right\|_{p}.
$$

Let $(a_n)_{n\in \N}$ and $(b_n)_{n\in \N}$ be two sequences in $  K$ such that $a_n<b_n$ and $\lim a_n=\lim b_n=l\in K$.  Let us denote by $t_n=b_n-a_n$. Fix $u:K\to \R^m$ continuous such that $L(t)=X_{u(t)}(f(t))$ for every $t\in K$.

By assumption $\frac{1}{t_n}\dsr\left(\e^{t_n X_{u(a_n)}}f(a_n),f(b_n)	\right)\rightarrow 0$. Applying Proposition~\ref{P:lim_dilation_backward}, let us prove that
$$
\delta_{1/t_n}^{f(a_n)}(f(b_n))
\xrightarrow[n \to +\infty]{} 
\e^{
\widehat{X}_{u(l)}}(0) .
$$

Since $\dsr\left(\e^{t_n X_{u(a_n)}}f(a_n),f(b_n)\right)\rightarrow 0$, there exists   sequence $(v_n)_{n\in\N}$ of controls, $v_n:[0,1]\rightarrow \R^m$, such that 
$$
|v_n|\leq 2 \dsr\left(f(b_n),\e^{t_n X_{u(a_n)}}f(a_n)\right)
$$
 almost everywhere on $[0,1]$ and 
$$
f(b_n)=\chronexp\int_0^1 X_{v_n(s)}\diff s \left(\e^{t_n X_{u(a_n)}}f(a_n)\right).
$$
Then
\begin{equation}\label{E:star}
\delta_{1/t_n}^{f(a_n)}(f(b_n))
=
\chronexp\int_0^1 {\delta_{1/t_n}^{f(a_n)}}_* X_{v_n(s)}\diff s 
\left(
\delta_{1/t_n}^{f(a_n)}
	\left(
		\e^{t_n X_{u(a_n)}}f(a_n)
	\right)
\right).
\end{equation}
Now, notice that
$$
 {\delta_{1/t_n}^{f(a_n)}}_* X_{v_n(s)}
=
\sum_{i=1}^{m} \frac{v_n^i(s)}{t_n} \left(t_n {\delta_{1/t_n}^{f(a_n)}}_* X_{i}\right)
$$
and that $\left(t_n {\delta_{1/t_n}^{f(a_n)}}_* X_{i}\right)_{n\in \N}$ locally uniformly converges towards $\widehat{X}_i$ on $\R^d$, while the $L^\infty$-norm of $\frac{v_n}{t_n}$ is upper bounded    by 
$$
\frac{2}{t_n}\dsr\left(f(b_n),\e^{t_n X_{u(a_n)}}f(a_n)\right),
$$ 
which converges toward $0$ almost everywhere on $[0,1]$. 

Hence, by continuity of the endpoint map with respect to the control,
$
\chronexp\int_0^1 {\delta_{1/t_n}^{f(a_n)}}_* X_{v_n(s)}\diff s
$
locally uniformly converges towards the identity.
On the other hand, 
$$
\delta_{1/t_n}^{f(a_n)}\left(\e^{t_n X_{u(a_n)}}f(a_n)\right)= \e^{t_n {\delta_{1/t_n}^{f(a_n)}}_{*}X_{u(a_n)}}(0).
$$
Again by local uniform convergence of $t_n {\delta_{1/t_n}^{f(a_n)}}_{*}X_{i}$ towards $\widehat{X}_{i}$ for all $1\leq i\leq m$ and by the convergence of $(u(a_n))_{n\in \N}$ towards $u(l)$,
$$
\delta_{1/t_n}^{f(a_n)}\left(\e^{t_n X_{u(a_n)}}f(a_n)\right)\rightarrow 
\e^{
\widehat{X}_{u(l)}}(0) ,
$$
so that \eqref{E:star} implies 
$$
\delta_{1/t_n}^{f(a_n)}(f(b_n))
\xrightarrow[n \to +\infty]{} 
\e^{
\widehat{X}_{u(l)}}(0) .
$$
\end{proof}

\subsection{Whitney condition on singular sub-Riemannian manifolds}\label{SS:singular_WC}
The aim of this section is to extend what we know about the $C^1_H$-Whitney condition to the case of  sub-Riemannian manifolds with singular points. This extension stands on the following result.

\begin{proposition}\label{P:lift}
Let $( \widetilde{M},\widetilde{\Delta}, \widetilde{g})$ be an equiregular lift of $(M,\Delta,g)$ on the open set $\Omega$. Let $(X_1,\dotsc,X_m)$ be a frame of $(M,\Delta,g)$ on $\Omega$ and  $(\widetilde{X}_1,\dotsc,\widetilde{X}_m)$ be a frame of $(\widetilde{M},\widetilde{\Delta},\widetilde{g})$ on $\widetilde{\Omega}$, the lift of $(X_1,\dotsc,X_m)$.  

Let $K\subset \R$ be compact and $(f,L):K\rightarrow \Omega\times TM$ be continuous.
If $(f,L)$ satisfies $\mathcal{P}_X$ on $K$, then there exists a continuous lift $(\widetilde{f},\widetilde{L}):K\rightarrow \widetilde{\Omega}\times T\widetilde{M}$ of $(f,L)$  that satisfies $\mathcal{P}_{\widetilde{X}}$ on $K$.
\end{proposition}

\begin{proof}
Let $(f,L):K\rightarrow \Omega\times TM$ be continuous and satisfying $\mathcal{P}_X$ on $K$. We construct a lift $(\widetilde{f},\widetilde{L}):K\rightarrow \widetilde{\Omega}\times T\widetilde{M}$ of $(f,L)$ by lifting a suitable absolutely continuous extension of $f$.

Let $u:\R\rightarrow \R^m$ be defined as follows. We set $\left.u\right|_{K}$ to be continuous and such that $X_{u(t)}=L(t)$ for all $t\in K$.  For every $(a,b)\subset K^c$ such that $a,b\in  K$, let $d(a,b)=\dsr\left( \e^{(b-a)X_{u(a)}}f(a),f(b)\right)$ and 
let
$$
M=2 \sup\left\{\frac{d(a,b)}{b-a}\mid(a,b)\subset K^c,a,b\in  K \right\}<\infty.
$$
Let $(a,b)\subset K^c$, $a,b\in  K$. Since $ b-\frac{d(a,b)}{M}>a$, we can define $u$ on $\R$ in the following way:
\begin{itemize}
\item if $t\in \left(a,b-\frac{d(a,b)}{M}\right)$, we set 
$$
u(t)=\frac{b-a}{b-a- d(a,b)/M}u(a),
$$

\item on $ \left[b-\frac{d(a,b)}{M},b\right)$ we   take $u$ measurable such that
$$
\chronexp\int_{b- d(a,b)/M}^b X_{u(s)}\diff s 
\left(   
\e^{(b-a)X_{u(a)}}f(a)
\right)=f(b).
$$
By definition of $M$, we can further assume that 
$$
|u(t)| \leq M \quad \text{ for all }t\in \left[b-\frac{d(a,b)}{M},b\right).
$$
\end{itemize}

On the non-compact components of $K^c$, we set $u$ to be such that $X_{u(t)}=L(b)$ if the component is of the form $(-\infty,b)$, and  $X_{u(t)}=L(a)$ if the component is of the form $(a,\infty)$.
Let us prove that for any $t_0,t\in K$, we have 
\begin{equation}\label{E:f=chronexp}
\chronexp\int_{t_0}^t  X_{u(s)}\diff s 
\left(   
 f(t_0)\right)=f(t).
\end{equation}

For all $t_0\in K$, there exists  an open neighborhood $O$ of $f(t_0)$, $\Phi:\bar{O}\rightarrow \R^d$ a smooth system of coordinates and an open interval $I\subset\R$ such that $t_0\in I$, $f(K\cap I)\subset O$ and for all $t\in I$, 
$\chronexp\int_{t_0}^t  X_{u(s)}\diff s 
\left(   
 f(t_0)\right)\in O$.
 
Let 
$$ 
f^*(t)=
\Phi \left( \chronexp\int_{t_0}^t  X_{u(s)}\diff s 
\left(    f(t_0) \right)\right).
$$
In order to prove \eqref{E:f=chronexp}, let us show that $\Phi(f(t))= f^*(t)$ for all $t\in K\cap I$.

We denote by $\|\cdot \|$ the Euclidean norm with respect to the coordinates $\Phi$. By continuity of $\Phi$, there exists $C>0$ such that 
$$
\| X_i \| \leq C,	 \qquad \forall q\in O, \forall 1\leq i\leq m
$$
and
$$
\left\|
		f(t)- \e^{(t-s)X_{u(s)}}f(s)
\right\|
\leq 
C\dsr\left(f(t),\e^{(t-s)X_{u(s)}}f(s)\right)\leq C \omega(|t-s|),
\qquad \forall t,s\in K\cap I.
$$
Hence for all $t,s\in I\cap K$,
$$
\begin{aligned}
	\left\|
		f(t)-f(s)
	\right\|
&\leq
	\left\|
		f(t)- \e^{(t-s)X_{u(s)}}f(s)
	\right\|
+	
	\left\|
		\e^{(t-s)X_{u(s)}}f(s)-f(s)
	\right\|
\\
&\leq
	C \left(\omega(|t-s|) +\|u\|_{\infty}|t-s|\right)
	\\
&\leq C'|t-s|,
\end{aligned}
$$
that is, $\Phi(f)$ is $C'$-Lipschitz continuous on $K\cap I$. Then $\Phi(f)$ admits a $C'$-Lipschitz continuous extension $\hat{f}$ on $ I$.

To show \eqref{E:f=chronexp}, we then show that $f^*$ and $\hat{f}$ coincide on $K\cap I$. Both are absolutely continuous and satisfy $f^*(t_0)=\hat{f}(t_0)$. Furthermore, the derivatives of $f^*$ and $\hat{f}$ are almost everywhere equal on $K$, and  for all $(a,b)\in K^c$, $a,b\in K$,
$$
\int_a^b{ f^*}'(s)\diff s = \Phi(f(b))-\Phi(f(a))=\hat{f}(b)-\hat{f}(a),
$$
since, by construction of $u$,
$$
\chronexp\int_{a}^b  X_{u(s)}\diff s 
\left(   
 f(a)\right)=\chronexp\int_{b-d(a,b)/M}^b X_{u(s)}\diff s 
\left(   
\e^{
(b-a) X_{u(a)}
}
\left(   
 f(a)\right)
\right)=f(b).
$$
Hence for all $t\in K\cap I$,
$$
f^*(t)=\int_{t_0}^t {f^*}'(s)\diff s=\int_{t_0}^t \hat{f}'(s)\diff s=\hat{f}(t).
$$

 Let $t_0\in K$ and let
$\widetilde{f}(t_0)$ be such that $\psi \left( \widetilde{f}(t_0) \right)= f(t_0)$, where $\psi$ is as in
Definition~\ref{D:lift}. We define the curve
$$
\gamma:\R\rightarrow \widetilde{M}
$$
such that for all $t\in\R$, 
$$
\gamma(t)=\chronexp\int_{t_0}^t \widetilde{X}_{u(s)}\diff s 
\left(   
\widetilde{f}(t_0)
\right).
$$
Then for all $t\in K$, we set $\widetilde{f}(t)=\gamma(t)$ and $\widetilde{L}(t)=\dot{\gamma}(t)=\widetilde{X}_{u(t)}(\gamma(t))$.

We claim that $(\widetilde{f},\widetilde{L})$ is a lift of $(f,L)$. 
By construction of $u$ and $\gamma$,
$$
\psi(\gamma(t))=\chronexp\int_{t_0}^t \psi_{*}\widetilde{X}_{u(s)}\diff s 
\left(   
\psi(\widetilde{f}(t_0))
\right)=\chronexp\int_{t_0}^t  X_{u(s)}\diff s 
\left(   
 f(t_0)\right)=f(t)
$$
for all $t\in K$ and, since $\dot{\gamma}(t)=\widetilde{X}_{u(t)}$, we have $\psi_{*}\widetilde{L}(t)=X_{u(t)}=L(t)$.

Let us now prove that such a lift satisfies the  Whitney condition. To alleviate notations in the following, we set $g=\widetilde{f}$ and $Y_i=\widetilde{X}_i$, $1\leq i\leq m$. Moreover, up to restricting $\widetilde{\Omega}$, we assume that there exists a continuously varying system of privileged coordinates 
$$
\Phi:
 (p,q) \longmapsto  \Phi_p(q)\in  \R^d,
$$ 
on  $\widetilde{\Omega}$.
As a consequence of Proposition~\ref{P:Wg<=>Wd} it is enough to show that
 for all $l\in K$, for all sequences $(a_n)_{n\in \N}$, $(b_n)_{n\in \N}$ in $K$ such that $a_n<b_n$ and $a_n,b_n\rightarrow l\in K$, we have 
\begin{equation}\label{E:backward_dilation_sufficient}
\lim_{n\to \infty} \delta^{ g(a_n)}_{\frac{1}{b_n-a_n}}(g(b_n))= \e^{\widehat{ Y}_{u(l)}}(0).
\end{equation}
For any interval  $(a,b)$ in  $\R$, by construction
$$
g(b)=\chronexp\int_{a}^b Y_{u(s)}\diff s(g(a)),
$$
thus, by reparametrizing,
$$
\delta^{ g(a)}_{\frac{1}{b-a}}(g(b))
=
\chronexp\int_{a}^b{\delta^{ g(a)}_{\frac{1}{b-a}}}_* Y_{u(s)}\diff s(0)
=
\chronexp\int_{0}^1(b-a){\delta^{ g(a)}_{\frac{1}{b-a}}}_* Y_{u(a+t(b-a))}\diff t(0).
$$
For any sequences
$(a_n)_{n\in \N}$, $(b_n)_{n\in \N}$ in $K$ such that $a_n<b_n$ and $a_n,b_n\rightarrow l\in K$, 
$
(b_n-a_n){\delta^{ g(b_n)}_{\frac{1}{b_n-an}}}_* Y_{i}\rightarrow \widehat{Y}_{i}
$ locally uniformly on $\R^d$, for all $1\leq i\leq m$. Hence to prove \eqref{E:backward_dilation_sufficient}, we now show that for $v_n(t)=u(a_n+t (b_n-a_n))$, $t\in [0,1]$,
$$
v_n\xrightarrow[n \to +\infty]{L^1((0,1),\R^m)}  u(l).
$$

For all $n\in \N$, let $K_n=(b_n-a_n)^{-1}(K-a_n)\cap[0,1] $. By uniform continuity of $\left.u\right|_K$ on compact subsets of $K$, $(\left.v_n\right|_{K_n})_{n\in\N} $ uniformly converges to $u(l)$. Regarding $(\left.v_n\right|_{K_n^c})_{n\in\N} $, as a first step, 
  let us compare $u$ to $u(l)$ on an interval $(a,b)\subset K^c$, $a,b\in K$.
  
For $t\in \left(a,b-d(a,b)/M\right)$, we have set $u(t)=\frac{b-a}{b-a-d(a,b)/M}u(a)$,
and 
 for $t\in \left[b-d(a,b)/M,b\right)$, we have imposed $|u(t)| \leq M$. 
Then
\begin{equation}\label{E:bound_connected_component}
\begin{aligned}
\int_a^b
|u(s)-u(l)|\diff s
&\leq
\left(
b-a-\frac{d(a,b)}{M}
\right)
\left|
	u(a)\frac{b-a}{b-a-d(a,b)/M}-u(l)
\right|
+
\frac{d(a,b)}{M}(M+|u(l)|),
\\
&
\leq
(b-a)
\left|u(a)-u(l)\right|+\omega(b-a)\left(1+\frac{2|u(l)|}{M }\right),
\end{aligned}
\end{equation}
where we used $\mathcal{P}_X$ for the inequality   $d\left( a,b\right)\leq \omega\left( b-a\right)$.

Since $K$ is a closed subset of $\R$, $K^c $ is a countable union of open intervals, notably  for all $n\in \N$ there exist $I(n)\subset \N$ and two countable (or finite) families of reals  $(c_n^k)_{k\in I(n)}$ and $(d_n^k)_{k\in I(n)}$  such that 
$$
K^c\cap (a_n,b_n) =\bigcup_{k\in I(n)} (c_n^k,d_n^k).
$$
Then 
$$
K_n^c\cap (0,1) =\bigcup_{k\in I(n)} \left(\frac{c_n^k-a_n}{b_n-a_n},\frac{d_n^k-a_n}{b_n-a_n}\right),
$$

and for all $n\in \N$, $k\in I(n)$, $( c_n^k, d_n^k)$ is a connected component of $K^c$,  hence a bound of type \eqref{E:bound_connected_component} holds.
Thus
$$
\begin{aligned}
\int_{\frac{c_n^k-a_n}{b_n-a_n}}^{\frac{d_n^k-a_n}{b_n-a_n}}
|v_n(t)-u(l)|\diff t
&=
\frac{1}{b_n-a_n}\int_{ c_n^k}^{ d_n^k}
|u(s)-u(l)|\diff s
\\
&\leq
\frac{(c_n^k-d_n^k)}{b_n-a_n}
\left|
u( c_n^k )-u(l)
\right|
 +\frac{\omega(d_n^k-c_n^k)}{(b_n-a_n)}\left(1+\frac{2|u(l)|}{M }\right),
\end{aligned}
$$
and 
$$
\begin{aligned}
\int_{K_n^c}|v_n(t)-u(l)|\diff t&=
\sum_{k\in I(n)} 
\int_{\frac{c_n^k-a_n}{b_n-a_n}}^{\frac{d_n^k-a_n}{b_n-a_n}}
|v_n(t)-u(l)|\diff t
\\&
\leq
\| \left.v_n\right|_{K_n}-u(l)\|_{\infty}
+
\left(1+\frac{2|u(l)|}{M}\right) \sum_{k\in I(n)} \frac{\omega \left(d_n^k-c_n^k\right) }{b_n-a_n}.
\end{aligned}
$$

As shown previously, $\| \left.v_n\right|_{K_n}-u(l)\|_{\infty}\rightarrow 0$. Regarding
$\sum_{k\in I(n)}\frac{\omega(d_n^k-c_n^k)}{b_n-a_n} $, recall that $\omega(t)=t\phi(t)$ with $\phi(t)\rightarrow 0$ as $t\rightarrow0^+$. Then
$$
 \frac{1}{b_n-a_n}\sum_{k\in I(n)}  \omega(d_n^k-c_n^k)
 <
 \frac{1}{b_n-a_n}  \sum_{k\in I(n)}  (d_n^k-c_n^k)\phi(d_n^k-c_n^k)
  <
  \sup_{[0,b_n-a_n]}\phi.
$$
In other terms,
$$
\sum_{k\in I(n)} \frac{\omega (d_n^k-c_n^k)}{b_n-a_n}\xrightarrow[n \to +\infty]{}0,
$$
and
$$
\int_0^1|v_n(t)-u(l)|\diff t=\int_{K_n\cap (0,1)}|v_n(t)-u(l)|\diff t+\int_{K_n^c\cap (0,1)}|v_n(t)-u(l)|\diff t\xrightarrow[n \to +\infty]{}0.
$$
\end{proof}

\begin{corollary}
Let $(X_1,\dots,X_m)$ and $(Y_1,\dots,Y_m)$ be two frames of the singular sub-Riemannian structure $(M,\Delta,g)$ on the open subset $\Omega\subset M$. Let $K\subset \R$ be compact and $(f,L):K\rightarrow \Omega\times TM$ be continuous. Then, as in the equiregular case, $(f,L)$ satisfies $\mathcal{P}_X$ if and only if $(f,L)$ on $K$ satisfies $\mathcal{P}_Y$ on $K$.
\end{corollary}
\begin{proof}
Let $c:\Omega\rightarrow \mathrm{O}(m)$  be a smooth map onto the orthogonal group such that 
$$
Y_i =\sum_{j=1} ^m c_{ij} X_j , \quad 1\leq i\leq m.
$$
(See Definition~\ref{D:frame}.) Without loss of generality, there exists an equiregular lift $(\widetilde{M},\widetilde{X}_1,\dots,\widetilde{X}_m)$  of the sub-Riemannian structure $(M,X_1,\dots ,X_m)$ on $\Omega$. We denote by $\psi:\widetilde{M}\rightarrow M$ the associated submersion.
Then let us define for all $q\in\widetilde{\Omega}$
$$
\widetilde{Y}_i(q)=\sum_{j=1} ^m c_{ij}(\psi(q)) \widetilde{X}_j(q), \quad 1\leq i\leq m.
$$
Then $(\widetilde{Y}_1,\dots,\widetilde{Y}_m)$
is a smooth frame for the sub-Riemannian manifold $(\widetilde{M},\widetilde{\Delta},\widetilde{g})$, and 
$$
\psi_{*}\widetilde{Y}_i (\psi(q))=\sum_{j=1} ^m c_{ij}(\psi(q)) \psi_*\widetilde{X}_j(\psi(q))=Y_i, \quad 1\leq i\leq m.
$$

Let $(f,L)$ be a curve in $\Omega$ satisfying $\mathcal{P}_X$ on $K$. By applying Proposition \ref{P:lift}, it can be lifted to a curve $(\widetilde{f},\widetilde{L})$ satisfying $\mathcal{P}_{\widetilde{X}}$ on $K$. 
Since the structure is regular on $\widetilde{\Omega}$, we have that $(\widetilde{f},\widetilde{L})$ satisfies $\mathcal{P}_{\widetilde{Y}}$ on $K$ by Proposition~\ref{P:real_Px<->Py}.
We conclude by noticing that $(f,L)$ must then satisfy $\mathcal{P}_Y$ on $K$. 
This is a direct consequence of the relation 
$$
\dsr
\left(
\widetilde{f}(t) , \e^{(t-s)  \widetilde{Y}_u}  \widetilde{f}(s)
\right)
\geq
\dsr
\left(
\psi\left(\widetilde{f}(t)\right),\psi\left(\e^{(t-s)\widetilde{Y}_u}\widetilde{f}(s)	\right)
\right)
=
\dsr
\left(
f(t) , \e^{(t-s)  Y_u}  f(s)  
\right),
$$
which holds for every $u\in\R^m$, 
as consequence of \eqref{E:proj_lift_dist}.
\end{proof}

As a consequence, Definitions~\ref{D:C1H_Whitney_cond} and \ref{D:C1H_extension_prop} for $\mathcal{C}^1_H$-Whitney condition and extension property are independent of the choice of the frame, and we have the following immediate corollary.

\begin{corollary}\label{C:projection_W_T}
Let $(M,\Delta,g)$ be a possibly singular sub-Riemannian manifold and let $(\widetilde{M},\widetilde{\Delta},\widetilde{g})$ be an equiregular lift of $(M,\Delta,g)$. If $(\widetilde{M},\widetilde{\Delta},\widetilde{g})$ has the $C^1_H$ extension property, then so does $(M,\Delta,g)$.
\end{corollary}

\section{A sufficient condition for the $C^1_H$ extension property}\label{S:Strong_pliability}
\subsection{Strong pliability}\label{SS:def_strong_pliability}

\begin{definition}[Strong pliability]\label{D:strong_pliability}
Let $(q, u)\in M\times\R^m$ and 
let $\mathbb{G}=(\R^d,(\widehat{X}_1,\dots,\widehat{X}_m))$ be a nilpotent approximation of $(M,\Delta,g)$ at $q$. Define the space $\mathcal{C}_0=\{v\in C^0([0,1],\R^m) \mid v(0)= 0\}$ and  the map 
$$
\begin{array}{rccc}
\mathcal{F}^{ u}:
&
\mathcal{C}_0
&
\longrightarrow
&
\mathbb{G}\times \R^m
\\
&
v
&
\longmapsto
&
\left(\chronexp\int_0^1 \widehat{X}_{u+v(s)}\mathrm{d}s(0_{\mathbb{G}}),
v(1)\right).
\end{array}
$$

The pair $(q,u)$ is said to be \emph{strongly pliable} if 
for all $\eta>0$ there exists $v\in \mathcal{C}_0$ such that $\|v\|_\infty<\eta$, $\mathcal{F}^u(v)=\mathcal{F}^u(0)$ and $\mathcal{F}^u$ is a submersion at $v$.
\end{definition}
A pair $(q,u)\in M\times \R^m$ is strongly pliable in particular when $\mathcal{F}^u$ is submersion at $0$. 
This definition relates to what has been called pliability of the vector $\widehat{X}_u$ in \cite{juillet_sigalotti}, \textit{i.e.}, the property that  $\mathcal{F}^u$ is locally open at $0$. Naturally,  if $(q,u)$ is strongly pliable then $\widehat{X}_u$ is pliable.
Recall also that if $(q,u)$ is pliable then the curve $[0,1]\ni t\mapsto \e^{tX_u}$ cannot be rigid in the sense of \cite{bryant_hsu_rigid}.

\subsubsection{Regular points of the endpoint map}
\label{SSS:regular}

For every $v\in \mathcal{C}_0$, set $F(v)=\chronexp \int_0^1\widehat{X}_{u+v(s)}\diff s (0_\mathbb{G})$ and $j(v)=v(1)$ so that $\mathcal{F}^u(v)=(F(v),j(v))$. We have $D_vj=j$ by linearity, so that $D_v\mathcal{F}^u=(D_v F , j)$.
With $P_t=\chronexp \int_0^t  \widehat{X}_{u+v(s)}\diff s$, $0\leq t \leq 1$, we have 
$
D_vF(w)
= \int_0^1
{P_1}_*{P_s^{-1}}_{*}
 \widehat{X}_{w(s)}
 \diff s(0_\mathbb{G})
$.
We use this characterization to evaluate the corank of $\mathcal{F}^u$ at $v$.

Let     $(\lambda,\mu)\in {\R^{d+m}}^*$. If $(\lambda,\mu)\in {\mathrm{Im}D_v\mathcal{F}^u}^{\perp} $ then
$$
\lambda
 \cdot
\int_0^1
{P_1}_*{P_s^{-1}}_{*}
 \widehat{X}_{w(s)}
 \diff s(0_\mathbb{G})
 +\mu \cdot w(1)=0 \qquad \forall w\in \mathcal{C}_0.
$$
Let us rewrite 
$$
\lambda
 \cdot
\int_0^1
{P_1}_*{P_s^{-1}}_{*}
 \widehat{X}_{w(s)}
 \diff s(0_\mathbb{G})
=
\sum_{i=1}^m\int_0^1w_i(s) \psi_i(s)\diff s=\langle w, \psi \rangle_{L^2((0,1),\R^m)}
$$
with $\psi_i(s)=\lambda \cdot {P_1}_*{P_s^{-1}}_{*}
 \widehat{X}_{i}$. Then for all $w\in \langle\psi \rangle^{\perp}$, the codimension $0$ or $1$ subspace of $\mathcal{C}_0$ orthogonal to $\psi$, we have that 
$$
(\lambda,\mu)\cdot D_v\mathcal{F}^u(w)=\mu \cdot w(1)=0.
$$
Having $\mu \cdot w(1)=0$ for all $w\in \langle\psi \rangle^{\perp}$ implies that $\mu=0$, since
$\left\{ w(1)\mid w\in \langle\psi \rangle^{\perp}\right\}=\R^m$.
Hence elements of $ {\mathrm{Im}D_v\mathcal{F}^u}^{\perp}$ are of the form $(\lambda,0)$ with  $\lambda\in {\mathrm{Im}D_vF}^{\perp}$, and regular values of $\mathcal{F}^u$ need only be regular values of $F$.

To study the regularity points of $F$, we introduce a more classical endpoint map, that is, the extension of $F$ to $L^{\infty}([0,1],\R^m)$:
$$
\begin{array}{rccc}
G:
&
L^{\infty}([0,1],\R^m)
&
\longrightarrow
&
\mathbb{G} 
\\
&
v
&
\longmapsto
&
\chronexp\int_0^1 \widehat{X}_{u+v(s)}\mathrm{d}s(0_{\mathbb{G}}). 
\end{array}
$$ 

\begin{lemma}\label{L:restriction}
The pair $(q,u)$ is strongly pliable if and only if for all $\eta>0$ there exists $v\in L^{\infty}([0,1],\R^m)$ such that $\|v\|_{L^\infty}<\eta$, $G(v)=G(0)$ and $G$ is a submersion at $v$.
\end{lemma}

\begin{proof}
If $(q,u)$ is strongly pliable, then for all $\eta>0$ there exists $v\in \mathcal{C}_0$ such that $\|v\|_{ \infty}<\eta$, $F(v)=F(0)$ and $F$ is a submersion at $v$. Since $G$ is an extension of $F$, the same conclusion follows by replacing $F$ by $G$.

Let now $\eta>0$ and pick $v\in L^{\infty}([0,1],\R^m)$  such that $\|v\|_{L^\infty}<\eta$, $G(v)=G(0)$ and $G$ is a submersion at $v$, {\it i.e.},
$
D_vG:
L^{\infty}([0,1],\R^m)
\rightarrow
\R^d
$
is surjective.  
By the remarks above, we are left to prove that there exists $w\in \mathcal{C}_0$ such that $\|w\|_{\infty}\leq 2 \eta$ and $F(w)=F(0)$ and $F$ is a submersion at $w$.

We have that $G\in C^1(L^\infty([0,1],\R^m),\R^d)$ where $L^\infty([0,1],\R^m)$ is endowed with the $L^2$-topology (see for instance  \cite[Proposition 5.1.2]{trelat2008controle}, \cite[Section 3]{trelat2000affine}).
Moreover, $\mathcal{C}_0$ is dense in $L^\infty([0,1],\R^m)$ for the $L^2$ topology.
The conclusion then follows from Lemma~\ref{lem:top-deg} in the Appendix, taking $V=L^\infty([0,1],\R^m)$ endowed with the  $L^2$-topology, $W=\mathcal{C}_0$, 
$\mathcal{F}=G$,  and
$\mathcal{F}_n=F$, $v_n=G(0)$ for every $n\in\N$.
\end{proof}

\begin{corollary}\label{C:not_strongly_pliable}
If $0$ is a regular value of $G$ then the pair $(q,u)$ is strongly pliable. 
\end{corollary}

\begin{remark}
An equivalent formulation of Corollary~\ref{C:not_strongly_pliable}, extending \cite[Section 6]{juillet_sigalotti} is that if $(q,u)$ is not strongly pliable, then 
$[0,1]\ni t \mapsto \e^{t \widehat{X}_u}(0_{\mathbb{G}})$ is an abnormal curve. We recall that a curve
$$
[0,T]\ni t\longmapsto \chronexp \int_0^{t} \widehat{X}_{v(s)}\diff s (0_{\mathbb{G}})
$$
is \emph{abnormal} if the map 
$$
\begin{array}{ccc}
L^{\infty}([0,T],\R^m)
&
\longrightarrow
&
\mathbb{G} 
\\
w
&
\longmapsto
&
\chronexp\int_0^T \widehat{X}_{w(s)}\mathrm{d}s(0_{\mathbb{G}})
\end{array}
$$ 
is singular at $v$.
\end{remark}

\begin{remark}\label{R:u=0}
For every sub-Riemannian manifold $(M,\Delta,g)$, for all $q\in M$, the pair $(q,0)$ is strongly pliable. Indeed, the regularity of $0$ for $G$ in the case $u=0$ is a consequence of Chow's theorem (see for instance \cite{ABB_nov_2016}). The proof straightforwardly extends to pairs $(q,u)$ such that $X_u(q)=0$.
\end{remark}

\subsubsection{Second order conditions}
\label{SSS:second_order}

When $0$ is a singular point for $G$, we can still give conditions ensuring strong pliability of $(q,u)$
in terms of classical optimality and rigidity conditions (see for instance \cite{Agrachev_Sarychev_Morse_rigidity}). The result of this discussion is summarized in Figure~\ref{F:diagram_curves}.

Let us recall some classical conditions for $G$ to have regular values $v\in L^{\infty}([0,1],\R^m)$  arbitrarily close to $0$ such that  $G(v)=G(0)$. 
Namely, it is sufficient for $[0,1]\ni t\mapsto\e^{tX_u}$ not to be the projection of a  Goh or a weak Legendre singular extremal (see \cite{Agrachev_Sarychev_Morse_rigidity}, \cite[Section 20.4]{Sachkov2004Control_theory_geometric_viewpoint}, \cite[Section 12.3]{ABB_nov_2016}).
These conditions summarize as follows.

\begin{proposition}\label{P:goh_and_legendre}
For $\lambda\in {\mathrm{Im}D_0F}^{\perp}\subset T^*_{\e^{\widehat{X}_u}(0_\mathbb{G})}\mathbb{G}$, $\lambda\neq 0$, and  for all $t\in [0,1]$, let 
\begin{equation}\label{E:def_lambda_t}
\lambda_t={\e^{(1-t)\widehat{X}_u}}^*\lambda.
\end{equation}
(In particular $\lambda_1=\lambda$.)
Let $B_G(\lambda,t)$ and $B_L(\lambda,t)$ be two bilinear forms on $\R^m$ defined by  
$$
B_G(\lambda,t; v_1,v_2)= \lambda_t \cdot
\left[
 \widehat{X}_{v_1},
 \widehat{X}_{v_2}
 \right],
 \qquad
 \forall v_1,v_2\in \R^m,
$$
and
$$
B_L(\lambda,t; v_1,v_2)= \lambda_t\cdot
\left[
	\left[ 
	\widehat{X}_{u},
	\widehat{X}_{v_1}
		 \right],
	\widehat{X}_{v_2}
\right],
 \qquad
 \forall v_1,v_2\in \R^m.
$$ 
If for all $\lambda\in {\mathrm{Im}D_0F}^{\perp}$, $\lambda\neq 0$, there exist some $t\in[0,1]$, $v_1,v_2\in \R^m$ such that either 
\begin{equation}\label{E:Goh}
B_G(\lambda,t;v_1,v_2)\neq 0 \tag{$\mathbf{G}$}
\end{equation}
or 
\begin{equation}\label{E:Legendre}
B_L(\lambda,t;v_1,v_1)< 0\quad 
 \tag{$\mathbf{L}$}
\end{equation}
then $(q,u) $ is strongly pliable.

\end{proposition}

\begin{proof}
It follows from Lemmas 20.7 and 20.8 of \cite{Sachkov2004Control_theory_geometric_viewpoint} that, as soon as 
$$
\mathrm{ind}_- \lambda \,\mathrm{Hess}_0 \,G\geq k \qquad \forall \lambda\in {\mathrm{Im}D_0G}^{\perp} ,\lambda\neq 0,
$$
with $k$ the corank of $G$ at $0$,
$G$ has regular values $v\in L^{\infty}([0,1],\R^m)$  arbitrarily close to $0$ such that  $G(v)=G(0)$.

Moreover,  for $\lambda\in {\mathrm{Im}D_0G}^{\perp}$, $\lambda\neq 0$, if there exist some $t\in[0,1]$, $v_1,v_2\in \R^m$ such that either 
$
B_G(\lambda,t;v_1,v_2)\neq 0
$
or 
$
B_L(\lambda,t;v_1,v_1)< 0
$
then $\mathrm{ind}_- \lambda \,\mathrm{Hess}_0 \,G=+\infty$ (\cite[Section 20.4]{Sachkov2004Control_theory_geometric_viewpoint}).

By smoothness of $G$ with respect to the $L^2$ topology and $L^2$-density of $\mathcal{C}_0$ in $L^{\infty}([0,1],\R^m)$, we have $\mathrm{Im}D_0G^\perp=\mathrm{Im}D_0F^\perp$. The conclusion then follows from Lemma~\ref{L:restriction}.
\end{proof}

\begin{figure}[ht]
\centering
\begin{tikzpicture}

\draw (0.,0.) rectangle (4.1,5.2) ;
\draw (2.1,4.6) node{Pliable};

\draw (.1,.1) rectangle (4.,4.) ;
\draw (2.1,3.4) node{Strongly pliable};

\draw (.3,.3) rectangle (3.9,2.8) ;
\draw (2.1,2.45) node{not Goh nor};
\draw (2.1,1.95) node{weak Legendre};

\draw (.4,.4) rectangle (3.8,1.6) ;
\draw (2.1,1) node{Soft};

\draw (0.2,0.2) -- (.2,2.9)-- (4.3,2.9)--(4.3,4)--(8.1,4)--(8.1,.2)--(4.2,.2) --(.2,.2);

\draw (6.3,3.4) node{Abnormal};

\draw (4.4,0.3) rectangle (8.,2.8) ;
\draw (6.2,2.2) node{Rigid};

\draw (4.5,.4) rectangle (7.9,1.6) ;
\draw (6.2,1.25) node{Goh and};
\draw (6.2,.75) node{strong Legendre};

\end{tikzpicture}
\caption{
Inclusion diagram of different classes of horizontal curves.
\label{F:diagram_curves}}

\end{figure}
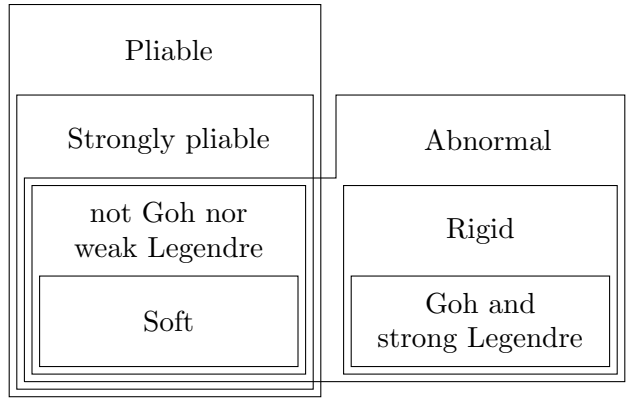
 In Figure~\ref{F:diagram_curves}, we represent the inclusion diagram of horizontal curves having properties related to strong pliability. 
 By Goh, we intend curves that are the projection of some $\lambda_t$ (as in \eqref{E:def_lambda_t}) such that $\lambda_1\in {\mathrm{Im}D_0F}^{\perp}$, $\lambda_1\neq 0$, and for all $t\in[0,1]$, $v_1,v_2\in \R^m$ , $B_G(\lambda_1,t;v_1,v_2)= 0$.
A Goh curve is strong (respectively, weak) Legendre if for all $t\in[0,1]$, $v\in \R^m$ , $B_L(\lambda_1,t;v,v) > 0$ (respectively, $B_L(\lambda_1,t;v,v) \geq 0$) (see \cite{Sachkov2004Control_theory_geometric_viewpoint}).
An example of non-Goh curves are soft abnormals, introduced in \cite[Definition 1]{boarotto_lerario_invisible_singular_curves}.

Let us present some consequences of Proposition~\ref{P:goh_and_legendre}. The following result is based on the fact that a non-Goh curve is strongly pliable.
\begin{proposition}\label{P:brackets_goh}
Let $q\in M$ and $\mathbb{G}=(\R^d,(\widehat{X}_1,\dots,\widehat{X}_m))$ be the nilpotent approximation of $(M,\Delta,g)$ at $q$. Let $\widehat{\Delta}$ be the distribution generated by $\widehat{X}_1,\dots,\widehat{X}_m$. Let $u\in \R^m$ be such that 
\begin{equation}\label{E:Goh_cond_deriv}
\sum_{k=0}^{+\infty} \left((\mathrm{ad}\widehat{X}_u)^k \widehat{\Delta}^2\right)_0= T_0\mathbb{G}.
\end{equation}
Then $(q,u)\in M\times \R^m$ is strongly pliable.
\end{proposition}
\begin{proof}
The proof works by proving that, under condition \eqref{E:Goh_cond_deriv},  the curve  $t\mapsto \e^{t\widehat{X}_u}(0)$ is not Goh. The conclusion then follows from Proposition~\ref{P:goh_and_legendre}.

Assume by contradiction that $\lambda_t$ is a  lift of the integral curve of $\widehat{X}_u$ such that $\lambda_1\in {\mathrm{Im}D_0F}^{\perp}$, $\lambda_1\neq 0$, $B_G(\lambda,t;v_1,v_2)=0$ for all $t\in[0,1]$ and all $v_1,v_2\in \R^m$. Thus by differentiating $t\mapsto B_G(\lambda,t;v_1,v_2)$  $k$ times and passing to the limit as $t\rightarrow 1^-$, we get 
$$
\lambda\cdot (\mathrm{ad}\widehat{X}_u)^k [\widehat{X}_{v_1},\widehat{X}_{v_2}]=0
\qquad \forall k\in \N,\forall v_1,v_2\in \R^m.
$$
If \eqref{E:Goh_cond_deriv} is satisfied, then $\lambda\in  \bigcap_{k=0}^{+\infty} ( (\mathrm{ad}\widehat{X}_u)^k \widehat{\Delta}^2)_0^\perp={T_0\mathbb{G}}^\perp=\{0\}$, hence the statement.
\end{proof}

We show below how Proposition~\ref{P:brackets_goh}  can be used to assess strong pliability for  step-2 distributions and more generally to medium-fat distributions (see \cite{rifford2014optimaltransport}).
This result extends \cite[Theorem 6.4]{juillet_sigalotti} where it was proved, as an application of \cite[Corollary 1.2]{bianchini_stefani_graded},  that for a step-2 Carnot group $\mathbb{G}$ every vector $\widehat{X}_{u}$, $u\in \R^m$, is pliable.

\begin{corollary}\label{P:step2_pliability}
Let $(M,\Delta,g)$ be a sub-Riemannian manifold and assume that $\Delta$ is medium-fat, \textit{i.e.}, for every $q\in M$ and every $u\in \R^m$ such that $X_u(q)\neq 0$,
\begin{equation}\label{E:medium_fat}
\Delta_{q}^2+\left[X_u,  \Delta^2\right]_{q}= T_{q}M .
\end{equation}
Then every pair $(q,u)\in M\times \R^m$ is strongly pliable.
\end{corollary}

\begin{proof}
The case where $X_u(q)=0$ follows from Remark~\ref{R:u=0}. 
Let now $q\in M$ and $u\in \R^m$ be such that $X_u(q)\neq 0$. Let $\Phi$ be a system of privileged coordinates at $q$.
Recall that for all positive integer $k>0$ $\widehat{\Delta}^k_0 ={\Phi}_*\Delta_q^k$ and that
\begin{equation}\label{E:bracket_at_0}
\left[\widehat{X}_u,\left[\widehat{X}_i,\widehat{X}_j\right]\right](0)
\in
\Phi_* \left[X_u,\left[X_i,X_j\right]\right](q)+\widehat{\Delta}^2_0,
\quad 
\text{ for every }
1\leq i,j\leq m.
\end{equation}
Assumption~\eqref{E:medium_fat} then implies that 
$
\widehat{\Delta}^2_0+ \left[\widehat{X}_u ,\widehat{\Delta}^2\right]_0= T_0\mathbb{G}
$. The conclusion follows from Proposition~\ref{P:brackets_goh}.
\end{proof}

 Another consequence of Proposition~\ref{P:goh_and_legendre} (in particular, of the property that a curve that is not weak Legendre is  strongly pliable) is the following.

\begin{corollary}\label{C:not_legendre}
Let $(M,\Delta,g)$ be a step-3 sub-Riemannian manifold.
Let $q\in M$ and $u\in \R^m$  be such that the convex positive cone
\begin{equation}\label{E:positive_convex}
C_{u} =
\mathrm{conv}
\left\{
W(q)
+
 [[X_u,V],V](q)
\mid 
V\in \Delta, W\in \Delta^2
\right\}
\end{equation}
is equal to $T_qM$.
Then $(q,u)$ is strongly pliable.
\end{corollary}

\begin{proof}
Let $\mathbb{G}=(\R^d,(\widehat{X}_1,\dots,\widehat{X}_m))$ be the nilpotent approximation of $(M,\Delta,g)$ at $q$ and $\widehat{\Delta}$ be the distribution generated by $\widehat{X}_1,\dots,\widehat{X}_m$. 
Then $\mathbb{G}$ is of step $3$ and 
\begin{equation}\label{E:positive_convex_nilp}
\widehat{C}_{u}=\mathrm{conv}
\left\{
W(0)
+
 [[\widehat{X}_u,V],V](0)
\mid 
V\in \widehat{\Delta}, W\in\widehat{\Delta}^2
\right\}
\end{equation}
is equal to $T_0\mathbb{G}$ (see \eqref{E:bracket_at_0}).

Following Proposition~\ref{P:goh_and_legendre}, we assume by contradiction that there exists  $\lambda\in {\mathrm{Im}D_0F}^{\perp}\setminus \{0\}$ such that $ \lambda\in \left(\widehat{\Delta}^2_{0}\right)^\perp$ and  
$$
\lambda \cdot  [[\widehat{X}_u,V],V](0)\geq 0
$$
for every  $V\in \widehat{\Delta}$. It then follows from the equality $\widehat{C}_u=T_0\mathbb{G}$ that $\lambda\cdot Z\geq 0$ for every $Z\in T_0\mathbb{G}$. This leads to a contradiction since $\lambda\neq 0$.
\end{proof}

\begin{example} Let $(x_1,x_2,x_3,y_1,y_2,w)$ be the canonical coordinates on $\R^6$ and define
$$
\left\{
\begin{array}{l}
X_1=\partial_{x_1},
\\
X_2=\partial_{x_2},
\\
X_3=\partial_{x_3}+x_1\partial_{y_1}+x_2\partial_{y_2}+\frac{1}{2}\left(x_1^2+\alpha x_2^2\right)\partial_w,
\end{array}
\right.
$$
with $\alpha<0$.
We set 
$$
\begin{array}{l}
W_1=\left[X_1,X_3\right]=\partial_{y_1}+x_1\partial_{w},
\\
W_2=\left[X_2,X_3\right]=\partial_{y_2}+\alpha x_2\partial_{w},
\\
Z=\partial_w,
\end{array}
$$
and we notice that
$$
\begin{array}{l}
\left[X_1,X_2\right]=\left[X_1,W_2\right]=\left[X_2,W_1\right]=0,
\\
\left[X_1,W_1\right]=Z,
\\
\left[X_2,W_2\right] =\alpha Z.
\end{array}
$$
By \cite[Theorem 4.2.10]{Bonfiglio2007Stratified}, the Lie algebra generated by $\{X_1,X_2,X_3\}$ is a Carnot algebra.

Take $u\in \R^3$ and let us prove the strong pliability of $(0,u)$. If $u=0$, this is a consequence of Remark~\ref{R:u=0}.  If either $u_1$ or $u_2$ is non-zero,  the  strong pliability of $(0,u)$ is a consequence of Proposition~\ref{P:brackets_goh} since 
$$
[X_u,[ X_1, X_3]]=u_1 Z,\quad [X_u,[ X_2, X_3]]=\alpha u_2 Z.
$$
Finally, 
if $u_1=0$, $u_2=0$ and $u_3\neq 0$,
$$
[[X_u, X_1 ],X_1 ]=u_3 Z,\quad [[X_u,  X_2], X_2]=  \alpha u_3 Z,
$$
and the strong pliability of $(0,u)$ follows from Corollary~\ref{C:not_legendre}.
\end{example}

\subsection{Strong pliability implies the $C^1_H$ extension property}\label{SS:Theorem}
\begin{theorem}\label{P:Strongly_pliable=>WT}
Let $(M,\Delta,g)$ be an equiregular sub-Riemannian manifold.
If every pair $(q,u)\in M\times \R^m$ is strongly pliable then the $C^1_H$ extension property holds for $(M,\Delta,g)$.
\end{theorem}

\begin{proof}
Let $(f,L)$ satisfy the Whitney condition on a closed set $K$. Without loss of generality, $K$ is compact, there exists a global frame $(X_1,\dots,X_m)$ of $\Delta$ on $M$, and we can rewrite $L=X_u$ with $u:K\rightarrow \R^m$ continuous. 
Let us define $\bar{f}$ on  $K^c$, which is a countable and disjoint union of open intervals.

Let $(a,b)\subset K^c$ be such that $a,b\in K$. For any $\eta>0$, we define the set $\mathcal{P}_\eta([a,b])\subset C^0([a,b],\R^m)$ of controls $v\in C^0([a,b],B_\eta(0))$ such that the integral curve of $X_{u(a)+v}$ is a $C^1_H ([a,b])$ extension of $f$ on $[a,b]$. In other words for $v\in \mathcal{P}_\eta([a,b])$ we have
\begin{align*}
&v(a)=0,
\\
&v(b)=u(b)-u(a),
\\
& \| v\|_{\infty}<\eta,
\\
&f(b)=\left(\chronexp\int_a^b X_{u(a)+v(s)}\diff s\right) f(a).
\end{align*}
Notice that $\mathcal{P}_\eta([a,b])\subset \mathcal{P}_{\eta'}([a,b])$ if $0< \eta\leq \eta'$ and define 
$$
\eta([a,b])=|b-a|+\inf \left\{\eta>0 \mid \mathcal{P}_\eta([a,b])\neq \emptyset\right\}.
$$
We claim that  $\inf \left\{\eta>0 \mid \mathcal{P}_\eta([a,b])\neq \emptyset\right\}<+\infty$, that is, there exists $w\in C^0([a,b],\R^m)$ such that $w(a)=u(a)$, $w(b)=u(b)$ and 
$$
f(b) = \left(\chronexp\int_a^b X_{w(s)}\diff s\right) f(a).
$$
This can be deduced, for instance, from Lemma~\ref{lem:top-deg} in the Appendix
taking as $V$ the space of piecewise continuous controls on $[a,b]$ endowed with the $L^2$ topology, $W=\{w\in C^0([a,b],\R^m)\mid w(a)=u(a),\,w(b)=u(b)\}$, 
$\mathcal{F}(v)= \left(\chronexp\int_a^b X_{v(s)}\diff s\right) f(a)$,  and
$\mathcal{F}_n=\mathcal{F}|_W$, $z_n=f(b)$ for every $n\in\N$. 
The existence of $v\in V$ such that 
$\mathcal{F}$ is a submersion at $v$ is a 
standard consequence of the Lie bracket-generating condition (see, e.g., \cite{Sussmann1976}).

Denoting by $(-\infty, \bar{b})$ and $(\bar{a},+\infty)$ the two unbounded components of $\R\setminus K$, we set $\bar{f}$ to be such that
$$
\bar{f}(t)=
\left\{
\begin{array}{ll}
\exp \left(\left(t-\bar{b}\right) X_{u(\bar{b})} \right)\left(f(\bar{b})\right)
&
\text{ if } t<\bar{b},
\\
\exp \left(\left(t-\bar{a}\right) X_{u(\bar{a})} \right)\left(f(\bar{a})\right)
&       
\text{ if } t>\bar{a}.
\end{array}
\right.
$$
We complete the extension $(\bar{f},\bar{u})$ of $(f,u)$ on $\R$ by taking for each $(a,b)\subset K^c$ with $a,b\in  K$
some  $v\in \mathcal{P}_{\eta([a,b])}([a,b])$ and setting 
$$
\bar{u}(t)=u(a)+v(t), \quad \bar{f}(t)=\left(\chronexp\int_a^t X_{u(a)+v(s)}\diff s\right) f(a), \quad \forall t\in [a,b].
$$

By construction, $\bar{f}:\R\rightarrow M$ is an extension of $f$ 
and, for every $t\in \R$ such that $u$ is continuous at $t$, the derivative 
$\dot{\bar{f}}(t)$ exists and is equal to $X_{u(t)}(f(t))$. We are left to prove that $u$ is continuous on $\R$. Notice that by construction $\left. u\right|_K$ and $\left. u\right|_{K^c}$ are continuous. We then focus on the continuity of $u$ at points of $\partial K$.

Take $\tau_{\infty}\in \partial K$ and a sequence $(\tau_n)_n\subset \R\setminus K$ such that $\tau_n\rightarrow \tau_{\infty}$. For every $n$, let  $(a_n,b_n)\subset K^c$ be such that $a_n,b_n\in K$ and  $\tau_n\in (a_n,b_n)$. 
If there exists a constant subsequence $((a_{n_k},b_{n_k}))_{k\in \N}$ of $((a_{n},b_{n}))_{n\in \N}$ then $\lim_{k\rightarrow \infty}u(\tau_{n_k})=u(\tau_\infty)$ by the continuity of $\left.u\right|_{[a_{n_k},b_{n_k}]}$. Assume then, without loss of generality, that $\lim a_n=\lim b_n=\tau_\infty$.

Since 
$$
\left| u(\tau_n)-u(a_n)\right|\leq \eta([a_n,b_n])
\text{ and }
u(a_n)\rightarrow u(\tau_\infty),
$$
the proof of the theorem is concluded by Lemma~\ref{L:limit_eta(an,bn)} below.
\end{proof}

\begin{lemma}\label{L:limit_eta(an,bn)}
Let $(a_n)_n$ and $(b_n)_n$ be two sequences in $\partial K$ such that $(a_n,b_n)\subset\R\setminus K$. If $\lim a_n=\lim b_n=\tau_\infty\in\R$  then $\eta([a_n,b_n])\rightarrow 0$.
\end{lemma}

\begin{proof}
To prove the lemma, we show that if $\lim a_n=\lim b_n < \infty$  then 
$$
\inf \left\{\eta>0 \mid \mathcal{P}_\eta\left(\left[a_n,b_n\right]\right)\neq \emptyset\right\}\rightarrow0, 
\quad \text{ as } n\rightarrow \infty.
$$
Equivalently, given $\eta>0$, we prove that there exists $N(\eta)$ such that if $n>N(\eta)$ then there exists  $v_n\in C^0([a_n,b_n],\R^d)$ such that 
\renewcommand\arraystretch{1.5}
\begin{equation}\label{E:sol_cond}
\left\{
\begin{array}{>{\displaystyle} l}
v_n(a_n)=0,
\\
v_n(b_n)=u(b_n)-u(a_n),
\\
 \| v_n\|_{\infty}<\eta,
\\
f(b_n)=\chronexp\int_{a_n}^{b_n} X_{u(a_n)+v_n(s)}\diff s \left( f(a_n)\right).
\end{array}\right.
\end{equation}

By applying Theorem~\ref{T:equiregular} at $f(\tau_\infty)$, we pick a neighborhood $\Omega$ of $f(\tau_\infty)$ and a  continuously varying system of privileged coordinates
$$
\begin{array}{rccc}
\Phi:&\Omega\times \Omega &\longrightarrow & \R^d
\\
& (x,y)&\longmapsto& \Phi_x(y).
\end{array}
$$ 
For $x\in\Omega$ and $\lambda>0$, we associate with $\Phi$ the quasi-homo\-geneous dilation $\delta_\lambda^x$. 

There exists a neighborhood $\mathcal{V}$  of $0$ in $\mathcal{C}_0 =\{w\in C^0([0,1],\R^m) \mid w(0)= 0\}$ such that 
$$
\chronexp \displaystyle{\int_0^t }  (b_n-a_n) X_{u(a_n)+w(s)}  \mathrm{d}s \left(f(a_n)\right)\in \Omega
$$
for all $t\in[0,1]$, $w\in \mathcal{V}$ and $n$ large enough. Hence, setting $t_n=b_n-a_n$,
we define the endpoint map 
$$
\begin{array}{rccl}
\mathcal{F}_{n}:
&
\mathcal{V}
&
\longrightarrow
& 
\R^d\times \R^m
\\
&
w
&
\longmapsto
&
\left(
\delta_{1/t_n}^{f(b_n)}
\left(
\chronexp \displaystyle{\int_0^1 }  t_n X_{u(a_n)+w(s)}   \mathrm{d}s \left(f(a_n)\right)\right),
u(a_n)+w(1)
\right).
\end{array}
$$ 
If $\mathcal{F}_{n}(w)=\left(0,u(b_n)\right)$ and $\|w\|_\infty<\eta$
then
$[a_n,b_n]\ni s\mapsto w\left(\frac{s-a_n}{b_n-a_n}\right)$ satisfies \eqref{E:sol_cond} (recall that $\delta_{1/t_n}^{f(b_n)}  ( f(b_n) )=0$).
Then we are left to prove that there exists   $N(\eta)$  such that if $n>N(\eta)$, there exists $w_n\in \mathcal{V}$ such that  $\|w_n\|_{\infty}<\eta$ and 
$$
\mathcal{F}_n(w_n)=\left(0,u(b_n)\right).
$$

Distributing the dilation we get
$$
\mathcal{F}_n(w)=
\left(
\chronexp \displaystyle{\int_0^1 } \left( t_n {\delta_{1/t_n}^{f(b_n)}}_* X_{u(a_n)+w(s)} \right) \mathrm{d}s \, 
\left(
\delta_{1/t_n}^{f(b_n)}  ( f(a_n) )
\right),
u(a_n)+w(1)\right).
$$
The Whitney condition ensures that $\delta_{1/t_n}^{f(b_n)}  ( f(a_n) )\rightarrow \e^{-\widehat{X}_{u(\tau_\infty)}}(0)$ (see Propositions \ref{P:lim_dilation_forward} and \ref{P:Wg<=>Wd}), and 
$t_n {\delta_{1/t_n}^{f(b_n)}}_* X_{u(a_n)+w(s)}$ is bounded and  locally uniformly converges towards 
$ \widehat{ X }_{u(\tau_\infty)+w(s)}   $. 
Hence 
$\mathcal{F}_n$ locally uniformly converges towards
$$
\begin{array}{rccl}
\mathcal{F}_{\infty}:
&
\mathcal{V}
&
\longrightarrow
& 
\R^d\times \R^m
\\
&
w
&
\longmapsto
&
\left(
\chronexp \displaystyle{\int_0^1 } \widehat{ X }_{u(\tau_\infty)+w(s)}   \mathrm{d}s \left(\e^{-\widehat{X}_{u(\tau_\infty)}}(0)\right),
u(\tau_\infty)+w(1)
\right).
\end{array}
$$ 
Let  $\mathbb{G}=(\R^d,(\widehat{X}_1,\dots,\widehat{X}_m))$ be the Carnot group structure of the nilpotent approximation of $(M,\Delta,g)$ at $f(\tau_\infty)$. Denote by $*$ its group operation, and recall that horizontal vector fields on $\mathbb{G}$ are left-invariant with respect to $*$.
Then
$$
\mathcal{F}_\infty(w)=
\left(
\left(
\e^{-\widehat{X}_{u(\tau_\infty)}}(0)
\right)
*\chronexp \displaystyle{\int_0^1 } \widehat{ X }_{u(\tau_\infty)+w(s)}   \mathrm{d}s (0),
u(\tau_\infty)+w(1)
\right).
$$
With $\psi (g,u)=\left(\left(
\e^{-\widehat{X}_{u(\tau_\infty)}}(0)
\right)*g,u(\tau_\infty)+u\right)$, which is a diffeomorphism from $\R^d\times\R^m$ onto itself, we have that
$$
\mathcal{F}_\infty=\psi\circ \mathcal{F}^{u(\tau_\infty)},
$$
where $\mathcal{F}^{u(\tau_\infty)}$ stands for the map introduced in Definition~\ref{D:strong_pliability}.

Therefore, by the strong pliability hypothesis, 
there exists  
$w_\eta$ in $\mathcal{V}$  such that $\|w_\eta\|<\eta/2$, 
$
\mathcal{F}_\infty(w_\eta)
=\mathcal{F}_\infty(0)
$,
 and  
$\mathcal{F}_\infty$ is a submersion at  $w_\eta$.

Notice that 
$$ 
\mathcal{F}_\infty(0)=
\psi\circ \mathcal{F}^{u(\tau_{\infty})}(0)
=
\psi \left(
\e^{\widehat{X}_{u(\tau_\infty)}}(0), 0\right)
=\left( 
\e^{-\widehat{X}_{u(\tau_\infty)}}(0)
*
\e^{\widehat{X}_{u(\tau_\infty)}}(0)
,u(\tau_\infty)\right)=(0,u(\tau_\infty)),	
$$ 
again by applying the $*$-left-invariance of $\widehat{X}_{u(\tau_\infty)}$.

It follows from Lemma~\ref{lem:top-deg} in the Appendix, with $V=W=\mathcal{V}$, $\mathcal{F}=\mathcal{F}_\infty$ and $z_n=(0,u(b_n))$ for $n\in\N$, that 
given $\eta>0$, there exists $N(\eta)>0$ such that for all $n> N(\eta)$
the equation 
$\mathcal{F}_n(w_n)=(0,u(b_n))$ has a solution 
$w_n$ with $\|w_n\|_\infty<\eta$. This concludes the proof of the lemma.
\end{proof}

\begin{corollary}\label{C:step_2}
All step-2 sub-Riemannian manifolds have the $C^1_H$ extension property.
\end{corollary}
\begin{proof}
If the manifold $(M,\Delta,g)$ is equiregular,  all pairs $(q,u)\in M\times \R^m$ are strongly pliable (see section~\ref{SSS:regular}, Corollary~\ref{P:step2_pliability}), hence the result by Theorem~\ref{P:Strongly_pliable=>WT}.

If the manifold $(M,\Delta,g)$ is not equiregular, it can   be locally lifted to a step-2 equiregular manifold (see {\it e.g.} \cite[Section 2.4]{jean2014control}). By Corollary~\ref{C:projection_W_T}, since  equiregular step-2 sub-Riemannian manifolds have the $C^1_H$ extension property, so does $(M,\Delta,g)$.
\end{proof}

\section{Lusin approximation of horizontal curves}\label{S:Lusin}

Let $(M,\Delta,g)$ be a sub-Riemannian manifold, and let $(X_1,\dots,X_m)$ be a frame of the distribution.
As a consequence of \cite[Theorem 2]{Vodopyanov_2006_differentiability_curves}, we have the following Rademacher-type theorem.
\begin{theorem} 
Let $(M,\Delta,g)$ be an equiregular sub-Riemannian manifold and $(X_1, \dots ,X_m)$ be a frame of the distribution.
Let  $\gamma:[a,b]\rightarrow M$ be  an absolutely 
continuous horizontal curve on $M$.

Let $\Phi$ be a continuously varying system of privileged coordinates.
For almost every $t\in [a,b]$ there exists $u\in \R^m$, such that
$$
\lim_{h\rightarrow  0} 
\frac{1}{h}
\widehat{\mathrm{d}}_{\mathrm{SR}}
\left(
\Phi_{\gamma(t)}(\gamma(t+h)),
 \e^{h\widehat{X}_u} (0)
\right)=0,
$$
where $\widehat{\mathrm{d}}_{\mathrm{SR}}$ is the Carnot-Caratheodory distance for the sub-Riemannian structure on $\R^d$ having $(\widehat{X}_1,\dots,\widehat{X}_d)$ as a frame, with   $\mathbb{G}=(\R^d,(\widehat{X}_1,\dots,\widehat{X}_d))$ the nilpotent approximation of $(M,\Delta,g)$ at  $\gamma(t)$.
\end{theorem}

We will use the following corollary.
\begin{corollary}\label{P:Rademacher}
Let $(M,\Delta,g)$ be a sub-Riemannian manifold. 
Let $[a,b]\subset \R$ and $\gamma:[a,b]\rightarrow M$ be an absolutely continuous  horizontal curve.
Then for almost every $t\in [a,b]$ there exists $u\in \R^m$ such that
\begin{equation}\label{E:Vodopyanov_coro}
\lim_{h\rightarrow  0} 
\frac{1}{h}
\dsr
\left(
\gamma(t+h),	\e^{h X_u}\gamma(t)
\right)=0.
\end{equation}
\end{corollary}
\begin{proof}
Let us first consider the equiregular case.
Let $t\in [a,b]$ be such that there exists $u(t)\in\R^m$ such that
$$
\lim_{h\rightarrow  0} 
\frac{1}{h}
\widehat{\mathrm{d}}_{\mathrm{SR}}
\left(
\Phi_{\gamma(t)}(\gamma(t+h)),
 \e^{h\widehat{X}_{u(t)}} (0)
\right)=0.
$$
Applying \cite[Theorem 7.32]{bellaiche1996TagentSpace} at $\gamma(t)$, there exist $\varepsilon>0$, $C>0$  such that, as soon as 
$$
\max (\dsr(\gamma(t),q),\dsr(\gamma(t),q'))\leq \varepsilon,
$$
we have
\begin{equation}\label{E:nilp_dist_vs_dist}
\dsr(q,q')
\leq 
\widehat{\mathrm{d}}_{\mathrm{SR}}
\left(
\Phi_{\gamma(t)}(q),
\Phi_{\gamma(t)}(q')
\right)
+
C
\widehat{\mathrm{d}}_{\mathrm{SR}}
\left(
0,
\Phi_{\gamma(t)}(q')
\right)
\widehat{\mathrm{d}}_{\mathrm{SR}}
\left(
\Phi_{\gamma(t)}(q),
\Phi_{\gamma(t)}(q')
\right)^{1/r}
\end{equation}
where $r$ is the step of $(M,\Delta,g)$.

For $|h|$ sufficiently small, $\Phi^{-1}_{\gamma(t)}\left(	\e^{h \widehat{X}_{u(t)}}(0)\right)$ is well defined and the triangular inequality yields
$$
\begin{aligned}
\dsr
\left(
\gamma(t+h),	\e^{h X_{u(t)}}\gamma(t)
\right)
\leq
& 
\dsr
\left(
\gamma(t+h),\Phi^{-1}_{\gamma(t)}\left(	\e^{h \widehat{X}_{u(t)}}(0)\right)
\right)
\\&+
\dsr
\left(
	\e^{h X_{u(t)}}\gamma(t),
	\Phi^{-1}_{\gamma(t)}\left(	\e^{h \widehat{X}_{u(t)}}(0)\right)
\right).
\end{aligned}
$$
As a consequence of \eqref{E:nilp_dist_vs_dist}, since 
$\widehat{\mathrm{d}}_{\mathrm{SR}}
\left(
0,
\e^{h \widehat{X}_{u(t)}}(0)
\right)
\leq  |h| |u(t)|$, in order to prove \eqref{E:Vodopyanov_coro} it is sufficient to have both
$$
\lim_{h\rightarrow  0} 
\frac{1}{h}
\widehat{\mathrm{d}}_{\mathrm{SR}}
\left(
\Phi_{\gamma(t)}(\gamma(t+h)),
 \e^{h\widehat{X}_{u(t)}} (0)
\right)=0
$$
and
$$
\lim_{h\rightarrow  0} 
\frac{1}{h}
\widehat{\mathrm{d}}_{\mathrm{SR}}
\left(
\Phi_{\gamma(t)}\left(	\e^{h X_{u(t)}}\gamma(t)\right),
 \e^{h\widehat{X}_{u(t)}} (0)
\right)=0.
$$
The first limit coincides with our assumption on $t$, and the second one is a consequence of two distance estimates for $h$ small enough.
First, from \cite[Propsition~7.26, Equation~(50)]{bellaiche1996TagentSpace},  
$$
\begin{aligned}
\frac{1}{C}
\widehat{\mathrm{d}}_{\mathrm{SR}}
\left(
\Phi_{\gamma(t)}\left(	\e^{h X_{u(t)}}\gamma(t)\right),
 \e^{h\widehat{X}_{u(t)}} (0)
\right)
\leq 
&
\left\|
\Phi_{\gamma(t)}\left(	\e^{h X_{u(t)}}\gamma(t)\right)-
 \e^{h\widehat{X}_{u(t)}} (0)
\right\|_{\gamma(t)}
\\
&+
\left\|
 \e^{h\widehat{X}_{u(t)}} (0)
\right\|_{\gamma(t)}^{1-1/r}
\left\|
\Phi_{\gamma(t)}\left(	\e^{h X_{u(t)}}\gamma(t)\right)-
 \e^{h\widehat{X}_{u(t)}} (0)
\right\|_{\gamma(t)}^{1/r}.
\end{aligned}
$$
Second, from
\cite[Theorem~2.3, Equation~(2.14)]{jean2014control},
$$
\left\|
\Phi_{\gamma(t)}\left(	\e^{h X_{u(t)}}\gamma(t)\right)-
 \e^{h\widehat{X}_{u(t)}} (0)
\right\|_{\gamma(t)}\leq C |u(t)|^{1+1/r} |h|^{1+1/r}.
$$
Combining the two,
$$
\widehat{\mathrm{d}}_{\mathrm{SR}}
\left(
\Phi_{\gamma(t)}\left(	\e^{h X_{u(t)}}\gamma(t)\right),
 \e^{h\widehat{X}_{u(t)}} (0)
\right)
\leq 
C 
|u(t)| |h|
\left(
|u(t)|^{1/r} |h|^{1/r}
+
|u(t)|^{1/r^2} |h|^{1/r^2}
\right),
$$
hence the result in the equiregular case.

If the manifold is not equiregular, as in  the proofs of Propositions~\ref{P:Taylor_exp} and Corollary~\ref{C:step_2}, we exploit the existence of local lifts of the sub-Riemannian structure that are equiregular. 
Consider an horizontal lift $\widetilde{\gamma}$ of $\gamma$. By the first part of the proof, we deduce that for almost every $t\in [a,b]$ there exists $u\in \R^m$ such that 
$$
\lim_{h\rightarrow  0} 
\frac{1}{h}
\dsrlift
\left(
\widetilde{\gamma}(t+h),	\e^{h \widetilde{X}_u}\widetilde{\gamma}(t)
\right)=0,
$$
where $\widetilde{X}_i$ is the lift of $X_i$ for every $1\leq i\leq m$.
Using \eqref{E:proj_lift_dist}, we deduce \eqref{E:Vodopyanov_coro}.
\end{proof}

Following  the classical scheme of proof for Lusin approximation  theorems (see \cite{le_donne_speight2016LusinStep2,speight2016LusinCarnot,juillet_sigalotti} for the case of Carnot groups), we give a version for general sub-Riemannian manifolds.

In the following we denote by $\mathcal{L}$ the Lebesgue measure on $\R$.
\begin{proposition}[Lusin approximation of an horizontal curve]\label{P:Lusin}
Let $(M,\Delta,g)$ be a sub-Riemannian manifold having the $C^1_H$ extension property and let $\gamma:[a,b]\rightarrow M$ absolutely continuous be an  horizontal curve. Then for any $\varepsilon>0$ there exists $K\subset [a,b]$ compact  with $\mathcal{L}([a,b]\setminus K)<\varepsilon$ and a curve $\gamma_1:[a,b]\rightarrow M$ of class $C_H^1$ such that  $\gamma$ and $\gamma_1$ coincide on $K$.
\end{proposition}

\begin{proof}
Let $\varepsilon >0$.
We want to prove that there exists a compact set $K\subset [a,b]$ with $\mathcal{L}([a,b]\setminus K)<\varepsilon$ such that the $C^1_H$-Whitney condition holds for $(\gamma,\dot{\gamma})$ on $K$.
The proposition then follows from the $C^1_H$ extension property.

By Corollary~\ref{P:Rademacher}, there exists $A\subset [a,b]$ of full measure such that, for any $t\in A$, the curve $\gamma$ admits an horizontal derivative at $t$, denoted by $X_{u(t)}(\gamma(t))$ using the local frame $(X_1,\dots,X_m)$. Moreover, the family of functions over $A$, $(f_h)_{h\in(0,1)}$, defined by
$$
f_h(t)=\frac{1}{h}
\dsr
\left(
	\gamma(t+h),\e^{h X_{u(t)}}\gamma(t)
\right)
$$
pointwise converges to $0$ as $h\rightarrow 0$. Applying the classical Lusin Theorem to the map $u:A\to \R^m$, there exists a compact set $K\subset A$ such that $u$ is uniformly continuous on $K$ and $\mathcal{L}(A\setminus K)<\varepsilon/2$. 
Furthermore, by Egorov's Theorem, we have the uniform convergence of $(f_h)_h$ towards $0$ on a compact subset $K'\subset K$ such that $\mathcal{L}(K\setminus K')<\varepsilon/2$.

This implies that there exist $\omega:\R^+\rightarrow\R^+$ and $K'$ such that $\omega(t)=o(t)$ at $0^+$, $\mathcal{L}([a,b] \setminus K')<\varepsilon$, and 
$$
\dsr\left(\gamma(t),\e^{(t-s)X_{u(s)}}\gamma(s)	\right)\leq \omega(|t-s|),\quad \forall t,s\in K'.
$$
\end{proof}

The above result finds its application in the study of $1$-countably rectifiable sets (see \cite{le_donne_speight2016LusinStep2}).
A set $E\subset M$ is said to be \emph{$1$-countably rectifiable} if there exists a countable family of Lipschitz curves $f_k:\R\rightarrow M$ such that
$
\mathcal{H}^1
\left(
E\setminus \cup_k f_k(\R)
\right)=0
$, 
where $\mathcal{H}^1$ denotes the 1-dimensional Hausdorff measure.
\begin{corollary}\label{C:rectifiability}
Let $(M,\Delta,g)$ be a sub-Riemannian manifold having the $C^1_H$ extension property and let $E $  be a $1$-countably rectifiable subset of $M$. Then there exists a countable family of $C^1_H$ curves $f_k:\R\rightarrow M$ such that 
$
\mathcal{H}^1
\left(
E\setminus \cup_k f_k(\R)
\right)=0
$.
\end{corollary}

\section*{Appendix}

Let us present here a useful technical result based on standard topological degree  considerations.

\begin{lemma}\label{lem:top-deg}
Let $V$ be a normed space and $W$ an affine and dense subspace of $V$. 
Fix $v\in V$ and let $\mathcal{F}\in C^1(V,\R^d)$ be a submersion at $v$. 
Consider a sequence of functions $\mathcal{F}_n:W\to \R^d$ with the property that $\mathcal{F}_n$ locally uniformly converges to $\mathcal{F}|_{W}$. Let, moreover, $z_n$ be a sequence in $\R^d$ converging to $\mathcal{F}(v)$. 
Then there exists a sequence $w_n$ in $W$ converging to $v$ in $V$ and such that, for $n$ large enough, 
$\mathcal{F}_n(w_n)=z_n$. 
\end{lemma}
\begin{proof}
By assumption there exist $\phi_1,\dots,\phi_d\in V$ such that the map
$$
\begin{array}{rccl}
\mathfrak{F}:
&
\R^d
&
\longrightarrow
& 
\R^d
\\
&
(x_1,\dots,x_d)
&
\longmapsto
&
\mathcal{F}(v+x_1 \phi_1+\dots+x_d\phi_d)
\end{array}
$$ 
is a local diffeomorphism at $0$. 

Let $v_n$ be a sequence in $W$ converging to $v$ in $V$. 
Denote by $W_L$ the linear space $\{w-w'\mid w,w'\in W\}$ and   
consider, for each $i=1,\dots,d$, a sequence $\varphi_i^n$ in $W_L$ converging to 
$\phi_i$ in $V$.  Then the sequence of maps 
$$
\begin{array}{rccl}
\mathfrak{G}_n:
&
\R^d
&
\longrightarrow
& 
\R^d
\\
&
(x_1,\dots,x_d)
&
\longmapsto
&
\mathcal{F}_n(v_n+x_1 \phi_1^n+\dots+x_d\phi_d^n)
\end{array}
$$ 
locally uniformly converges to $\mathfrak{F}$. 

Let $r>0$ be small enough so that 
the restriction of $\mathfrak{F}$ to the ball $B_r(0)$ of center the origin and radius $r$ is a diffeomorphism between $B_r(0)$ and $\mathfrak{F}(B_r(0))$. 
Then $\mathfrak{G}_n|_{\overline{B_r(0)}}$ uniformly converges to $\mathfrak{F}|_{\overline{B_r(0)}}$. Hence,  
for any $K$ compactly contained in $\mathfrak{F}(B_r(0))$ and  for $n$ large enough, 
the topological degree $d(\mathfrak{G}_n,B_r(0),z)$ is equal to $1$ or $-1$ for every $z\in K$. In particular, choosing $K=\mathfrak{F}(\overline{B_{\rho r}(0)})$ with $\rho\in (0,1/2)$ and replacing $r$ by $2\rho r$ in the above argument, we have that 
for $n$ large enough, there exists $x^n\in B_{2\rho r}(0)$ such that 
$$\mathcal{F}_n(v_n+x_1^n \phi_1^n+\dots+x_d^n\phi_d^n)=\mathfrak{G}_n(x^n)=z_n.$$

In order to recover the convergence to $v$ of the sequence 
$v_n+x_1^n \phi_1^n+\dots+x_d^n\phi_d^n$ it suffices to 
notice that if $z_n$ is in $\mathfrak{F}(\overline{B_{\rho r}(0)})$, then $x^n$ can be chosen of norm smaller than $2\rho r$. The conclusion then follows from the convergence of $v_n$ to $v$ and the uniform boundedness of $\{\phi_j^n\mid j=1,\dots,d,\,n\in\N\}$. 
\end{proof}

\bibliographystyle{alphaabbr}
\bibliography{biblio}

\begin{thebibliography}{BCGS13}

\bibitem[ABB16]{ABB_nov_2016}
A.~A. Agrachev, D.~Barilari, and U.~Boscain.
\newblock {\em Introduction to Riemannian and Sub-Riemannian geometry}.
\newblock Preprint SISSA, November 2016.

\bibitem[ABL17]{boarotto_lerario_invisible_singular_curves}
A.~A. Agrachev, F.~Boarotto, and A.~Lerario.
\newblock Homotopically invisible singular curves.
\newblock {\em Calc. Var. Partial Differential Equations}, 56(4):56:105, 2017.

\bibitem[AG78]{agrachevGamkre_chron}
A.~A. Agrachev and R.~V. Gamkrelidze.
\newblock Exponential representation of flows and a chronological enumeration.
\newblock {\em Mat. Sb. (N.S.)}, 107(149)(4):467--532, 639, 1978.

\bibitem[AS87]{Agrachev_Sarychev_filtrations}
A.~A. Agrach\"ev and A.~V. Sarychev.
\newblock Filtrations of a {L}ie algebra of vector fields and the nilpotent
  approximation of controllable systems.
\newblock {\em Dokl. Akad. Nauk SSSR}, 295(4):777--781, 1987.

\bibitem[AS96]{Agrachev_Sarychev_Morse_rigidity}
A.~A. Agrachev and A.~V. Sarychev.
\newblock Abnormal sub-{R}iemannian geodesics: {M}orse index and rigidity.
\newblock {\em Ann. Inst. H. Poincar\'e Anal. Non Lin\'eaire}, 13(6):635--690,
  1996.

\bibitem[AS04]{Sachkov2004Control_theory_geometric_viewpoint}
A.~A. Agrachev and Y.~L. Sachkov.
\newblock {\em Control theory from the geometric viewpoint}, volume~87 of {\em
  Encyclopaedia of Mathematical Sciences}.
\newblock Springer-Verlag, Berlin, 2004.
\newblock Control Theory and Optimization, II.

\bibitem[BBS16]{IHP-Vol}
D.~Barilari, U.~Boscain, and M.~Sigalotti, editors.
\newblock {\em Geometry, analysis and dynamics on sub-{R}iemannian manifolds.
  {V}ol. 1 \& 2}.
\newblock EMS Series of Lectures in Mathematics. European Mathematical Society
  (EMS), Z\"urich, 2016.
\newblock Lecture notes from the IHP Trimester held at the Institut Henri
  Poincar\'e, Paris and from the CIRM Summer School ``Sub-Riemannian Manifolds:
  From Geodesics to Hypoelliptic Diffusion'' held in Luminy, Fall 2014.

\bibitem[BCGS13]{MR3010287}
U.~Boscain, G.~Charlot, R.~Ghezzi, and M.~Sigalotti.
\newblock Lipschitz classification of almost-{R}iemannian distances on compact
  oriented surfaces.
\newblock {\em J. Geom. Anal.}, 23(1):438--455, 2013.

\bibitem[Bel96]{bellaiche1996TagentSpace}
A.~Bella\"iche.
\newblock The tangent space in sub-{R}iemannian geometry.
\newblock In {\em Sub-{R}iemannian geometry}, volume 144 of {\em Progr. Math.},
  pages 1--78. Birkh\"auser, Basel, 1996.

\bibitem[BH93]{bryant_hsu_rigid}
R.~L. Bryant and L.~Hsu.
\newblock Rigidity of integral curves of rank {$2$} distributions.
\newblock {\em Invent. Math.}, 114(2):435--461, 1993.

\bibitem[BLU07]{Bonfiglio2007Stratified}
A.~Bonfiglioli, E.~Lanconelli, and F.~Uguzzoni.
\newblock {\em Stratified {L}ie groups and potential theory for their
  sub-{L}aplacians}.
\newblock Springer Monographs in Mathematics. Springer, Berlin, 2007.

\bibitem[BS90]{bianchini_stefani_graded}
R.~M. Bianchini and G.~Stefani.
\newblock Graded approximations and controllability along a trajectory.
\newblock {\em SIAM J. Control Optim.}, 28(4):903--924, 1990.

\bibitem[FSSC01]{franchi_2001_rectifiability}
B.~Franchi, R.~Serapioni, and F.~Serra~Cassano.
\newblock Rectifiability and perimeter in the {H}eisenberg group.
\newblock {\em Math. Ann.}, 321(3):479--531, 2001.

\bibitem[Gru70]{Grushin}
V.~V. Gru{\v{s}}in.
\newblock A certain class of hypoelliptic operators.
\newblock {\em Mat. Sb. (N.S.)}, 83 (125):456--473, 1970.

\bibitem[Jea14]{jean2014control}
F.~Jean.
\newblock {\em Control of nonholonomic systems: from sub-{R}iemannian geometry
  to motion planning}.
\newblock SpringerBriefs in Mathematics. Springer, Cham, 2014.

\bibitem[JS17]{juillet_sigalotti}
N.~Juillet and M.~Sigalotti.
\newblock {Pliability, or the {W}hitney extension theorem for curves in
  {C}arnot groups}.
\newblock {\em Analysis \& PDE}, 10:1637--1661, 2017.

\bibitem[LDS16]{le_donne_speight2016LusinStep2}
E.~Le~Donne and G.~Speight.
\newblock Lusin approximation for horizontal curves in step 2 {C}arnot groups.
\newblock {\em Calc. Var. Partial Differential Equations}, 55(5):Art. 111, 22,
  2016.

\bibitem[Rif14]{rifford2014optimaltransport}
L.~Rifford.
\newblock {\em Sub-{R}iemannian geometry and optimal transport}.
\newblock SpringerBriefs in Mathematics. Springer, Cham, 2014.

\bibitem[SC16]{serracassano:IHPVol1}
F.~Serra~Cassano.
\newblock Some topics of geometric measure theory in carnot groups.
\newblock In D.~Barilari, U.~Boscain, and M.~Sigalotti, editors, {\em Geometry,
  analysis and dynamics on sub-{R}iemannian manifolds. {V}ol. 1}, EMS Series of
  Lectures in Mathematics, pages vi+324. European Mathematical Society (EMS),
  Z\"urich, 2016.

\bibitem[Spe16]{speight2016LusinCarnot}
G.~Speight.
\newblock Lusin approximation and horizontal curves in {C}arnot groups.
\newblock {\em Rev. Mat. Iberoam.}, 32(4):1423--1444, 2016.

\bibitem[Sus76]{Sussmann1976}
H.~J. Sussmann.
\newblock Some properties of vector field systems that are not altered by small
  perturbations.
\newblock {\em J. Differential Equations}, 20(2):292--315, 1976.

\bibitem[Tr{\'e}00]{trelat2000affine}
E.~Tr{\'e}lat.
\newblock Some properties of the value function and its level sets for affine
  control systems with quadratic cost.
\newblock {\em J. Dynam. Control Systems}, 6(4):511--541, 2000.

\bibitem[Tr{\'e}05]{trelat2008controle}
E.~Tr{\'e}lat.
\newblock {\em Contr{\^o}le optimal}.
\newblock Math{\'e}matiques Concr{\`e}tes. [Concrete Mathematics]. Vuibert,
  Paris, 2005.
\newblock Th{\'e}orie \& applications. [Theory and applications].

\bibitem[Vod06]{Vodopyanov_2006_differentiability_curves}
S.~K. Vodopyanov.
\newblock Differentiability of curves in the category of {C}arnot manifolds.
\newblock {\em Doklady Mathematics}, 74(2):686--691, 2006.

\bibitem[VP06]{Vodopyanov_2006_Whitney_Carnot_groups}
S.~K. Vodopyanov and I.~M. Pupyshev.
\newblock {W}hitney-type theorems on the extension of functions on {C}arnot
  groups.
\newblock {\em Sibirsk. Mat. Zh.}, 47(4):731--752, 2006.

\bibitem[Zim18]{ZimmermanWhitneyHeisenberg}
S.~Zimmerman.
\newblock The {W}hitney extension theorem for {$C^1$}, horizontal curves in the
  {H}eisenberg group.
\newblock {\em J. Geom. Anal.}, 28(1):61--83, 2018.

\end{thebibliography}

\end{document}